\documentclass[reqno,11pt]{amsart}
\usepackage{amsmath,amsthm,amssymb,amsfonts}
\usepackage{epsfig,color,graphicx,enumerate,bbm,stmaryrd}
\usepackage{accents}

\def\ds{\displaystyle}
\def\eps{{\varepsilon}}
\def\N{\mathbb{N}}
\def\R{\mathbb{R}}

\def\A{\mathcal{A}}

\def\Dr{D}
\def\Om{\Omega}
\def\B{\mathcal{B}}

\def\G{\mathcal{G}}
\def\HH{\mathcal{H}}

\def\vf{\varphi}

\newcommand{\be}{\begin{equation}}
\newcommand{\ee}{\end{equation}}
\newcommand{\bib}[4]{\bibitem{#1}{\sc#2: }{\it#3. }{#4.}}
\newcommand{\cp}{\mathop{\rm cap}\nolimits}

\newcommand{\ind}{\mathbbm{1}}
\numberwithin{equation}{section}
\theoremstyle{plain}
\newtheorem{teo}{Theorem}[section]
\newtheorem{lemma}[teo]{Lemma}

\newtheorem{prop}[teo]{Proposition}
\newtheorem{deff}[teo]{Definition}
\theoremstyle{remark}
\newtheorem{oss}[teo]{Remark}

\begin{document}

\title{Regularity of  minimizers of shape optimization problems involving  perimeter}

\author{Guido De Philippis}
\address{\textit{G.~De Philippis:} SISSA, Via Bonomea 265, 34136 Trieste, Italy.}
\email{guido.dephilippis@sissa.it}
%\footnote{CNRS, ENS Lyon}}

\author{Jimmy Lamboley}
\address{\textit{J.~Lamboley:} Universit\'e Paris-Dauphine, PSL Research University, CNRS, CEREMADE, 75016 Paris, France}
\email{lamboley@ceremade.dauphine.fr}

\author{Michel Pierre}
\address{\textit{M.~Pierre:} ENS Rennes, IRMAR, UBL, av Robert Schuman, 35170 Bruz, France}
\email{michel.pierre@ens-rennes.fr}

\author{Bozhidar Velichkov}
\address{\textit{B.~Velichkov:} Tour IRMA, 51 rue des Mathematiques, B. P. 53, 38041 Grenoble, France}
\email{bozhidar.velichkov@imag.fr}
%\footnote{Tour IRMA, 51 rue des Mathematiques, B. P. 53, 38041 Grenoble, France}}
%\address{Scuola Normale Superiore di Pisa, Piazza dei Cavalieri 7, 56126 Pisa, ITALY}

%\date{February 14, 2011}

\begin{abstract}
We prove existence and regularity of optimal shapes for the problem
$$\min\Big\{P(\Omega)+\mathcal G (\Omega):\ \Omega\subset D ,\ |\Omega|=m\Big\},$$
where $P$ denotes the perimeter, $|\cdot|$ is the volume, and the functional $\mathcal{G}$ is either one of the following:
\begin{itemize}
\item the Dirichlet energy $E_f$, with respect to a (possibly sign-changing) function $f\in L^p$;
\item a spectral functional of the form $F(\lambda_{1},\dots,\lambda_{k})$, where $\lambda_k$ is the $k$th eigenvalue of the Dirichlet Laplacian and $F:\R^k\to\R$ is Lipschitz continuous  and increasing in each variable.
\end{itemize}
The domain $D$ is the whole space $\R^d$ or a bounded domain. We also give general assumptions on the functional $\mathcal{G}$ so that the result remains valid.
\end{abstract}

\maketitle

%\textbf{Keywords:} Shape optimization, spectral cost, drop problems, Dirichlet energy

%\textbf{2010 Mathematics Subject Classification:} 49A50, 49G05, 49R50, 49Q10

%%%%%%%%%%%%%%%%%%%%%%%%%%%%%%%%%%%%%%%%%%%%%%%%%%

\section{Introduction}

This paper is concerned with the question of existence and regularity of solutions to shape optimization problems of the form
\begin{equation}\label{eq:pb}
\min\Big\{J(\Om)\ :\ \Om\in\A\Big\},
\end{equation}
where $\A$ is a class of domains in $\R^d$ (where $d\geq 2$) and $J:\A\to\R$ is a given shape functional. 
%We are mainly interested in the regularity of the optimal sets, but we also deal with the question of existence.
 We focus on the case where $J$ can be decomposed as the sum $P+\G$ of the perimeter $P$ and of a functional $\G$ which depends on the solution of some PDE defined on \(\Omega\). We find general assumptions  on $\G$ so that any minimizer for \eqref{eq:pb} in the class $\A=\{\Om\subset\R^d, |\Om|=m\}$ is a quasi-minimizer of the perimeter and therefore is $C^{1,\alpha}$ up to a residual set of codimension bigger than 8. Our hypotheses allow to deal with several functionals $\Om\mapsto\G(\Om)$ involving elliptic PDE and eigenvalues with Dirichlet boundary conditions on $\partial\Om$.\\

\noindent{\bf State of the art:}\\

The first  question one has to handse is the  existence of a minimizer $\Om^*\in\A$ for \eqref{eq:pb}. This step crucially relies on the choice of a suitable topology on $\A$ and this usually forces to relax the initial natural class $\A$ to a wider oner  of possibly irregular domains. For the functionals we are going to deal with here, we will often choose $\A$ to be a subclass of measurable sets. %the sets of finite perimeter.

Once existence is known the second question to handle concerns regularity of optimal shapes. Indeed  one usually  expects  the  optimal domain $\Omega^*$ for \eqref{eq:pb} to  more smooth than what a priori provided by the existence theory. The first step toward reaching this smoothness is usually very difficult, especially because one has to work with domains $\Omega^*$ whose boundary may even not be (locally) the graph of a function. Once it is known  that $\partial \Om^*$ is locally the graph of a (say Lipschitz continuous) function $\varphi$, it is reasonably easy  in many cases to write the first order optimality condition for \eqref{eq:pb} in terms of $\varphi$. It generally leads to a PDE system satisfied by $\varphi$. Then using nontrivial, but well-known results from PDE regularity theory, we may use bootstrap regularity arguments and reach high smoothness for $\varphi$, see Remark \ref{rk:reg}. Hence,  the most difficult step  is to gain regularity from scratch, namely to show that  the optimal shape $\Om^*$, which a priori enjoys very littele  regularity, is actually a Lipschitz or a $C^{1,\alpha}$ domain.

The most important example in this framework comes from the question of minimizing the perimeter, defined as $P(\Om)=\HH^{d-1}(\partial\Om)$ when $\Om$ is smooth (see Section \ref{sect:prelim} for a suitable relaxation of this definition), under volume constraint. Of course, the well-known isoperimetric inequality asserts that the ball is the unique minimizer for this problem, if it is admissible, and in that case of course, the regularity  is trivial. But in more general situations, for example for  the constrained isoperimetric problem
\begin{equation}\label{eq:isop}
\min \{P(\Om)\ :\ \;|\Om|=m, \;\;\Om\subset D\}
\end{equation}
where $D$ is a box in $\R^d$ too narrow to contain a ball of volume $m$, the regularity issue is not trivial. In this case, it can be proved that, if $D$ is bounded, an optimal shape $\Om^*$ exists in the class of sets of finite perimeter and that $\partial\Om^*\cap D$ is smooth (locally analytic) if $d\leq 7$, and in general is smooth up to a closed residual set of codimension bigger  than 8, see for example \cite{GMT83,giusti,maggi}.

 This has been generalized in many ways and led to the notion of {\em quasi-minimizer of the perimeter}. This means for $\Om^*$ that there exists $C\in\R$, $\alpha\in(d-1,d]$ and $r_{0}>0$ such that for every ball $B_{r}$ with $r\leq r_{0}$,
\begin{equation}\label{eq:quasi}
P(\Om^*)\leq P(\Om)+Cr^{\alpha}, \;\;\;\;\;\forall\;\; \Om\textrm{ such that }\Om\Delta\Om^*\subset B_{r}\cap D,
\end{equation}
%where $P(\Om,B_{r})(=\HH^{d-1}(\partial\Om\cap B_{r})$ when $\Om$ is smooth) denotes the relative perimeter in $B_{r}$
(see again Section \ref{sect:prelim}). This implies that $\Om^*$ enjoys strong regularity properties, namely
%if $\Om^*$ satisfies \eqref{eq:quasi} for \alpha)\in(0,\infty)\times(d-1,d]$ then
\begin{equation}\label{eq:reg}
\textrm{ the reduced boundary }\partial^*\Om^*\cap D\textrm{ is }C^{1,(\alpha-d+1)/2}\;\;\textrm{ and }\;\;\textrm{dim}_{\HH}((\partial\Om\setminus\partial^*\Om)\cap D)\leq d-8.
\end{equation}
Here $\partial\Om$ is the measure theoretical boundary of $\Om$ which coincides with  the topological boundary of $\Om$ for a suitable representative, see Section \ref{ssect:representative}.\\

Another class of energy functionals of great interest is related to elliptic PDE's with Dirichlet boundary conditions on $\partial\Om$. As a seminal example, we introduce the Dirichlet energy $E_{f}$ 
\begin{equation}\label{eq:dirichlet}
E_{f}(\Om):=\min\left\{\frac{1}{2}\int_{\Om}|\nabla u|^2\,dx-\int_{\Om}fu\,dx\ :\  u\in H^{1}_{0}(\Om)\right\},
\end{equation}
where $f$ is a fixed function of $L^2(D)$.
This is naturally defined for any open set $\Om$ of finite volume. But the class of open sets is not suitable for the  existence theory  and one has to introduce the concept of quasi-open set, see \cite{hepi05} and Section \ref{ssect:representative}.  Therefore, one considers the problem
\begin{equation}\label{eq:Ef}
\min \{E_{f}(\Om)\ :\  \Omega\textrm{ quasi-open, }|\Om|=m, \;\;\Om\subset D\}.
\end{equation}
In this case, and if $f\in L^\infty(D)$, it can be  proved that there exists an optimal shape which is actually an open set $\Om^*$ (see \cite{BHP}). Moreover, if  $f$ is nonnegative and $d=2$, it can be shown  that $\partial\Om^*$ is smooth (analytic), see \cite{briancon}. If \(d>2\), it is only known that  $\partial\Om^*$ is smooth up to a set of codimension bigger than 1, see \cite{briancon}. 
%There are indications that this estimate is not optimal, but in any case, even in a simplified blow-up case, the maximal dimension for full regularity is not yet understood; also, some results about the case where $f$ is not assumed to be nonnegative are given in \cite{briancon}, but one expects singularities of the boundary, even in dimension 2). 
The main argument in \cite{briancon} is based on the connection of Problem \eqref{eq:dirichlet} to a free-boundary type problem, and the regularity theory relies on the techniques introduced by   Alt and Caffarelli in \cite{altcaf}. This strategy strongly uses that $E_{f}(\Om)$ has a variational formulation  as a minimization over a class of functions $u\in H^1(D)$ 
%(functions of $H^1_{0}(\Om)$ are seen in $H^1(D)$ using an extension by 0 outside $\Om$), 
and that  the optimal shape $\Om^*$ is then the set of positivity of the optimal $u$. 

Note that, if $f$ changes sign, then $\partial \Omega^*$ will have singularities around each point where the optimal $u$ changes sign. This happens even in dimension two where the singularities are of cusps type, see e.g. \cite{Des}. This shows that the regularity of the optimal shapes is a difficult question in the present framework. And it is interesting to notice that, adding the perimeter in the energy to be minimized like we do here, does bring enough regularity for the optimal shapes even for signed data $f$ as proved later in this paper.

In \cite{landais} (see also \cite{ACKS}), the regularity of minimizers is investigated for the problemÉ
\begin{equation*}\label{eq:P+Ef}
\min \{P(\Om)+E_{f}(\Om)\ :\ |\Om|=m, \;\;\Om\subset D\},
\end{equation*}
where both of the previous functionals are involved. The main result there asserts that if  $f$ is nonnegative and in $L^\infty(D)$, then an optimal shape $\Om^*$ for problem \eqref{eq:isop} is a quasi-minimizer for the perimeter in the sense of \eqref{eq:quasi}, and therefore satisfies the regularity \eqref{eq:reg}. In the more general case where  $f\in L^q(D)$, with $q>d$ and \(f\ge 0\), or  $f\in L^\infty(D)$ with no assumption on its sign, it was proved in \cite{landais2} that the state function $u_{\Om^*}$ (i.e. the function achieving the  minimum in \eqref{eq:Ef}) was locally $C^{0,1/2}$ in $D$.  This clearly  implies that $\Om^*$ is an open set, but is not sufficient to conclude that $\Om^*$ is a quasi-minimizer of the perimeter (it gives $\alpha=d-1$ in \eqref{eq:quasi}). By a completely different strategy we will prove later in this paper that this is actually the case, see the proof of Theorem~\ref{th:mainD}.

Another class of interesting functionals is related to the spectrum of the Dirichlet-Laplacian over $\Om$, denoted $0<\lambda_{1}(\Om)\leq \lambda_{2}(\Om)\leq \cdots\leq \lambda_{k}(\Om)\leq\ldots$. The problem
\begin{equation}\label{eq:lk}
\min \{\lambda_{k}(\Om)\ :\  |\Om|=m,\  \Om\subset D\},
\end{equation}
where $k\in\N^*$, has  received a lot of attention in the last years. For the  particular  case $D=\R^d$,  only recently  a satisfying existence result has been proved in the class of quasi-open sets, see \cite{bulbk} and  \cite{mp}. 
In particular in \cite{bulbk},  even though existence was the main purpose, the author proves along the way some qualitative properties of optimal shapes, namely that they are bounded and of finite perimeter; its strategy led to the notion of sub- and super-solution for shape optimization problems.
% which appears to be new and useful in this framework.
  Let us stress  that for minimizers of  \eqref{eq:lk},  regularity is yet not understood except for $k=1$, see \cite{brla}.  There are however  some partial results, see \cite{bucmaprave}. In the recent work \cite{deve}, the first and last authors  studied a slightly different related problem, namely
\begin{equation}\label{eq:lkp}
\min \{\lambda_{k}(\Om)\ :\  P(\Om)=p\}.
\end{equation}
Making good use of the concept of sub/super-solution, they take again advantage of the presence of the perimeter and they were able to prove that solutions of \eqref{eq:lkp} are quasi-minimizer of the perimeter, and therefore they satisfy \eqref{eq:reg}. In particular, their strategy allows to prove regularity of optimal shapes for functionals for which the minimization problem   cannot be translated into a free boundary problem and for which the state function can change sign.\\

\noindent{\bf New results:}\\

Our main purpose here is to generalize the ideas of \cite{deve} in order to deal with problems of the form
\begin{equation*}\label{eq:P+Ef2}
\min \{P(\Om)+E_{f}(\Om)\ :\  |\Om|=m,\; \Om\subset D\;\;\textrm{ or }\;\;\;\min \{P(\Om)+\lambda_{k}(\Om)\ :\  |\Om|=m,\; \Om\subset D\}.
\end{equation*}
Our main result, Theorem \ref{th:mainD} below, proves  existence of  minimizers, and  that they  are quasi-minimizer of the perimeter (therefore satisfying \eqref{eq:reg}, see Theorem \ref{tama}).
In particular, comparing to the results of \cite{landais}, we strongly relax the assumptions on $f$ for the Dirichlet-energy case.  Namely we are able to deal with every   $f\in L^q(D)$ for $q\in(d,\infty]$ without any assumption on the sign.
Concerning  the case of eigenvalues, while the strategy of \cite{landais} (based on a free boundary formulation) could only be applied to the case $k=1$, we are able to deal with every $k$.
To obtain these results, we generalize the concepts of sub/super-solutions to the case of volume constraint, see Definitions \ref{subsoldef} and \ref{supsoldef}. In particular, we obtain two independent results for sub- and super-solutions, which are of complete different nature, and are interesting on their own. We refer to the beginning  of Sections \ref{sect:sup} and \ref{sect:sub}, respectively, for the statement of these results. Here we state the main consequence of these two statements, which, combined with a penalization procedure, lead to the main theorem of this paper.

\begin{teo}\label{th:mainD}
Suppose that $\Dr\subset\R^d$ is a bounded open set of class $C^2$ or the entire space $\Dr=\R^d$.
Then there exists a solution of the problem 
\begin{equation}\label{eq:P+Ef3}
\min \Big\{P(\Om)+\mathcal{G}(\Omega)\ :\  \Omega\text{ open},\ \Om\subset D,\ |\Om|=m\Big\},
\end{equation}
where $m\in(0,|\Dr|)$ and $\mathcal{G}$ is one of the following functionals:
\begin{itemize}
\item $\mathcal{G}=E_f$, where $f\in L^p(D)$ with $p\in(d,\infty]$ if $\Dr$ is bounded and $p\in (d,\infty)$ if $\Dr=\R^d$;
\item $\mathcal{G}=F(\lambda_{1},\cdots,\lambda_{k})$, where $F:\R^k\to\R$ is increasing in each variable and locally Lipschitz continuous.
\end{itemize}
Moreover, every solution $\Om^*$ of \eqref{eq:P+Ef3} is bounded and it is a quasi-minimizer of the perimeter with exponent $d-d/p\textrm{ or }d$ respectively, and therefore satisfies \eqref{eq:reg}.
%, and in the case $\mathcal{G}=F(\lambda_{1},\cdots,\lambda_{k})$ with $F$ of class $C^1$, $\partial^*\Om^*\cap D$ is $C^\infty$. {\Rd ?locally analytic?}}
\end{teo}

An interesting fact in the proof of the above Theorem, is that the proofs of existence and regularity are actually linked. Indeed:
\begin{itemize}
\item We will  prove existence for a related but different problem (see Proposition \ref{prop0}), and conclude that these solutions also solve \eqref{eq:P+Ef3} because they are smooth enough.
\item In the proof of existence for this related problem, namely \eqref{eq:P+Ef4}  in order to study minimizing sequences, one a priori has to  prove  that solutions  are  bounded. This relies on a density estimate which is a first step in the regularity theory, see Section \ref{ssect:boundedness}.
\end{itemize}

\begin{oss}
Let us note that the assumptions of Theorem \ref{th:mainD} are essentially sharp for what concerns existence of optimal sets, as the following examples show:
\begin{itemize}
\item For $f\in L^\infty(\R^d)$, existence of minimizers of  \eqref{eq:P+Ef3} could fail if $\mathcal{G}= E_f$ (or similarly of  \eqref{eq:P+Ef4} if $\mathcal{G}=\widetilde{E}_{f}$). For example, let   $f$ be  such that $0\leq f<1$ and $f(x)\to_{|x|\to\infty} 1$. Then the infimum of \eqref{eq:P+Ef3} equals $P(B)+E_{1}(B)$ where $B$ is a ball of volume $m$, and it is not attained. Indeed, by   symmetrization, for every set $\Om$ of volume $m$, we have
$$P(\Om)+E_{f}(\Om)> P(\Om)+E_{1}(\Om)\geq P(B)+E_{1}(B),$$ 
while a sequence of balls of volume $m$ that goes to $\infty$ achieves equality in  the limit.
\item There exists   a smooth convex unbounded box $D$ such that Problem \eqref{eq:P+Ef3} with $\mathcal{G}=\lambda_{1}$ has no solution. For example, take 
$$D=\left\{(x,y)\in(0,\infty)\times\R,\;  y^2<\frac{x}{x+1}\right\}\subset\R^2, \;\;\;\;\;\textrm{ and }\;\;\;\;m=|B(0,1)|=\pi.$$
Note that  \(D\)  does not contain any ball of volume $m$, though it almost does at the limit $x\to\infty$. Using the isoperimetric and Faber-Krahn inequalities,  one easily sees that, for every set $\Om\subset D$ of volume $m$,
$$P(\Om)+\lambda_{1}(\Om)> P(B_{1})+\lambda_{1}(B_{1}),$$
while equality is achieved for a sequence of sets converging to the ball at infinity.
\end{itemize}
\end{oss}

\begin{oss}
With similar notation, we could also consider the problem:
\begin{equation}\label{eq:<=m}
\min \Big\{P(\Om)+\mathcal{G}(\Omega)\ :\  \Omega\text{ open},\ \Om\subset D,\ |\Om|\leq m\Big\}.
\end{equation}
In general, this problem is not equivalent to Problem \eqref{eq:P+Ef3}. This can be easily seen by considering  the problem of minimizing $P+\lambda_{1}$ among all sets in $\R^d$ (in particular with no volume constraint): the solutions are balls (symmetrization) whose radius is the unique minimizer of $r\mapsto P(B_{1})r^{N-1}+\lambda_{1}(B_{1})r^{-2}$ (scaling of the functional). For any value $m$ bigger than the volume of those balls, it is clear that Problems \eqref{eq:P+Ef3} and \eqref{eq:<=m} have different solutions.\\
However, all the conclusions of the previous theorem are still valid for solutions of \eqref{eq:<=m}. To see this, one just needs to take into account  the following two remarks:
\begin{itemize}
\item The existence proof from Section \ref{sect:existence} can be repeated {\em verbatim} in the case of  \eqref{eq:<=m}.
\item A solution $\Om^*$ of Problem \eqref{eq:<=m} is a solution of Problem \eqref{eq:P+Ef3} if we replace $m$ by $|\Om^*|$. 
\end{itemize}
\end{oss}

\begin{oss}\label{rk:reg}
Once $C^{1,\alpha}$-regularity of the reduced boundary is obtained, one may wonder about higher regularity. In the case
 $\mathcal{G}=E_{f}$ with $f$ smooth enough, this is done classically by writing an optimality condition for problem \eqref{eq:P+Ef3}.  Namely one can show that in a weak sense,
 $$\mathcal{H}-\frac{1}{2}|\nabla u|^2=\mu\;\;\;\;\;\textrm{ on }\;\;\partial^*\Om^*,$$
where $\mathcal{H}$ is the mean curvature, $\mu\in\R$ is a Lagrange multiplier for the volume constraint, and $u$ is the state function,
% ($u\in H^1_{0}(\Om^*)$ such that $-\Delta u=f$ in $\Om^*$). 
A simple  bootstrap argument shows that if $f\in C^{k,\beta}(D)$ then $\partial^*\Om^*\cap D$ is $C^{k+3,\beta}$, see \cite{landais}.

A similar statement for $\mathcal{G}=F(\lambda_{1},\cdots,\lambda_{k})$ is more involved as eigenvalues may not be differentiable if they are multiple and thus  it is not straightforward to write an optimality condition. However, as it is noticed in \cite{bogoud}, this can still be done, at least assuming a priori smoothness.  In  \cite{bogosel},  a weak sense is given to this optimality condition and it is proved that the reduced boundary is $C^\infty$ when $F$ is smooth enough.
\end{oss}

{\bf Strategy of the proof and organization of the paper:}\\

The proof of the main result is carried out in several steps and there is a different section dedicated to each one of them. 

\begin{itemize}
\item {\bf Extending the admissible class of domains.} Our goal is to find a minimizer for the functional $\mathcal F$, which is a priori defined in the class of open sets. From the point of view of existence theory, it is more appropriate to consider classes of domains that are as large as possible. For this purpose, we define a functional $\widetilde{\mathcal F}$ on the class of Lebesgue measurable sets in $\R^d$. We notice that $\widetilde{\mathcal F}$ is not an extension of $\mathcal{F}$ but satisfies the inequality 
\begin{equation}\label{mainpropertyofF}
\widetilde{\mathcal{F}}(\Omega)\le \mathcal F(\Omega),\quad\hbox{for every open set}\quad\Omega\subset\R^d,
\end{equation}
while the equality holds for sets which are sufficiently regular. 
The construction of $\widetilde{\mathcal F}$ will be carried out in Section \ref{sect:prelim}, along basic facts and tools which will be used in the rest of the paper.\\
\item {\bf Existence of a minimizer of $\widetilde{\mathcal F}$.} The existence of an optimal domain is well known in the case where the admissible class is restricted to the family of measurable sets contained in a given set $\Dr\subset\R^d$ of finite measure, see Section \ref{sect:existence}. In the case where $\Dr=\R^d$,  in order to show existence of minimizers, we need to prove some qualitative properties of solutions, namely boundedness. This  will be  done in Section \ref{sect:sub}, while existence is proved in Section \ref{sect:existence}.\\
\item {\bf Penalization of the volume constraint:} This is a new difficulty compared to the result of \cite{deve}. In order to develop a regularity theory, we need to explain how minimizers for Problem \eqref{eq:P+Ef3} (or also \eqref{eq:P+Ef4}) are also solutions of an optimization problem with no constraint on the volume.  This will be obtained through a  penalization technique. In Section \ref{sect:pen}, we prove  a general  result  by assuming very  weak  properties on the functional $\mathcal{F}$, Lemma \ref{genlam06}, and we then show that these properties are satisfied by our functionals.\\
\item {\bf Regularity of the minimizers of $\widetilde{\mathcal F}$.} 
We generalize in  Sections \ref{sect:sup} and Section \ref{sect:sub} the notion of sub/supersolution from \cite{deve} for functionals with a volume term. We state two general results, Propositions \ref{prop2} and \ref{prop3}, which lead to  the desired regularity result for minimizers of $\widetilde{\mathcal F}$. 
 Compared to the results of \cite{landais,landais2} (where the author studies the regularity, but does not obtain a complete result when $f$ has no sign),
the main new idea it  is to prove that the torsion function $w_{\Om^*}$ (instead of the state $w_{\Om^*,f}$) is Lipschitz continuous (see the notation in Section \ref{sect:prelim}), and then to show that the variation of $E_{f}$ is controlled by the variation of $E_{1}$. In particular, this allows to avoid the use of the Monotonicity Lemma of Caffarelli-Jerison-Kenig \cite{CJK}.\\

\item {\bf Conclusion.} The previous steps show that  there exists a minimizer $\widetilde \Omega$ of  $\widetilde{\mathcal F}$ which is sufficiently regular. In particular  $\mathcal{F}(\widetilde\Omega)=\widetilde{\mathcal F}(\widetilde\Omega)$. Hence  by \eqref{mainpropertyofF}, $\widetilde \Omega$ is also a minimizer of $\mathcal F$ in the class of open sets. Using once again the results of Section \ref{sect:sup} and Section \ref{sect:sub}, we will prove that, if $\Omega$ is a minimizer of $\mathcal{F}$, then the set $\Omega^{(1)}$ of points of Lebesgue density one is again a minimizer and it  is regular.
\end{itemize}

\begin{oss}
It is clear from the above description that our strategy of proof strongly relies on the presence of a perimeter  term in  the functional we aim to minimize. Indeed, all the regularity issue boils down in showing that the optimal shapes are quasi-minimizers of the perimeter. In this respect the main step consists  in proving  Lipschitz continuity of the state function \(w_{\Omega^*}\) since it makes the term 
\[
E_1(\Omega^*)=-\frac{1}{2}\int |\nabla w_{\Omega^*}|^2
\]
behaving as a {\em volume term} and thus of lower order with respect to the perimeter.

One might wonder what can be said if one puts a constraint both on the measure {\em and} on the perimeter. For instance if one considers as in~\cite{MvDB} the problem
\begin{equation}\label{fakepb}
\min \Big\{\lambda_k(\Omega): |\Omega|\le m\quad P(\Omega)\le p\Big\},
\end{equation}
for given \(m,p>0\). In this situation the regularity issue is highly not trivial, at least when the perimeter constraint is not saturated. Indeed  in this case it is easy  to see that, understanding the regularity of solution of~\eqref{fakepb} is  equivalent to understanding the regularity  of solutions of~\eqref{eq:lk}, which is at the moment completely open when \(k\ge 2\). 
\end{oss}

%In the next Section, we start with the main definitions, in particular we focus on the suitable definitions of the perimeter $P$ and the energies $E_{f}$ and $\lambda_{k}$ for nonsmooth domains, in Section \ref{sect:sup} we investigate the regularity properties of super-solutions for the functional $P+\alpha|\cdot|$, in Section \ref{sect:sub} we focus on regularity properties of sub-solutions of $P+\G+\alpha|\cdot|$ where $\G$ satisfies suitable properties, mainly that its variations are comparable to the variation of the Dirichlet energy $E_{1}$. We conclude in Section \ref{sect:pen} and \ref{sect:final} with a proof of Theorem \ref{th:main} which relies on the results of Sections \ref{sect:sup} and \ref{sect:sub} and a penalization procedure to deal with the volume constraint.

\section{Preliminaries}\label{sect:prelim}

In this section, we review the notion of Sobolev space, Dirichlet energy and Dirichlet eigenvalues for sets that are only measurable. We also recall a few basic facts about sets of finite perimeter that are needed in this paper, and we conclude with some compactness and semi-continuity properties.

Given $D$ a measurable set in $\R^d$, we denote $\B(D)$ the class of measurable subsets of $D$.

\subsection{The Sobolev space $H^1_0(\Omega)$ and the Sobolev-like space $\widetilde H^1_0(\Omega)$}
Suppose first that $\Omega\subset\R^d$ is an open set. The Sobolev space $H^1_0(\Omega)$ is defined, as usual, as the closure of the smooth functions with compact support in $\Omega$, $C^\infty_c(\Omega)$, with respect to the Sobolev norm $\ds \|u\|_{H^1}^2=\|\nabla u\|_{L^2}^2+\|u\|_{L^2}^2$.

For a given Lebesgue measurable set $\Omega\subset\R^d$, we define the Sobolev-like space $\widetilde H^1_0(\Omega)$ as 
$$\widetilde H^1_0(\Omega)=\Big\{u\in H^1(\R^d)\ :\ u=0\quad\hbox{a.e. on}\quad \R^d\setminus \Omega\Big\}.$$
We notice that this space is also a Hilbert space, as it is closed in $H^1(\R^d)$. Moreover, if $\Omega\subset\R^d$ has finite Lebesgue measure ($|\Omega|<\infty$), then the inclusion $\widetilde H^1_0(\Omega)\subset L^2(\Omega)$ is compact.

%the functions $u\in \widetilde H^1_0(\Omega)$ are automatically extended by zero outside $\Omega$. 
\begin{oss} 
If $\Omega\subset\R^d$ is an open set, then clearly  $H^1_0(\Omega)\subset\widetilde H^1_0(\Omega)$. In general this  inclusion is strict, a typical example being  $\Omega=B_1\setminus \{(x_{1},\hat{x})\in\R\times\R^{d-1}\ :\ x_{1}=0\}$. Nevertheless, if $\Omega$ is a Lipschitz domain, then the two spaces coincide $H^1_0(\Omega)=\widetilde H^1_0(\Omega)$. More generally,  this is true if  $\Omega$ satisfies an exterior density estimate, see for example  \cite{deve} and Lemma \ref{olddens06}.  
\end{oss}

%\begin{oss}
%If $\Omega\subset\R^d$ is of finite Lebesgue measure ($|\Omega|<\infty$), then the inclusion $\widetilde H^1_0(\Omega)\subset L^2(\Omega)$ is compact. 
%\end{oss}

\subsection{Elliptic problems on measurable sets}\label{ssect:elliptic}
If $\Omega\subset\R^d$ is of finite Lebesgue measure, then for any $f\in L^2(\Omega)$, there is a unique minimizer in $\widetilde H^1_0(\Omega)$ of the functional 
$$J_f(u)=\frac12\int_{\R^d} |\nabla u|^2\,dx-\int_{\R^d} uf\,dx,$$
which we denote by $w_{\Omega,f}$ or simply by  $w_\Omega$ if $f\equiv1$. Writing the Euler-Lagrange equations for $w=w_{\Omega,f}$, we get
\begin{equation}\label{eulagw}
\int_{\R^d}\nabla w\cdot\nabla \varphi\,dx=\int_{\R^d} \varphi f\,dx,\quad\hbox{for every}\quad \varphi\in \widetilde H^1_0(\Omega).
\end{equation}
We will say that $w$ is the (weak) solution of the equation
\begin{equation}\label{eqmeassets}
-\Delta w=f\quad\hbox{in}\quad\Omega,\qquad w\in \widetilde H^1_0(\Omega).
\end{equation}
\noindent{\bf Estimate in $H^1$:} 
Testing \eqref{eulagw} with $\varphi=w$ we get
\begin{equation}\label{testwithsol}
\int_{\R^d}|\nabla w|^2\,dx=\int_{\R^d}f w\,dx.
\end{equation}
By using  that  $\lambda_1(\Omega)|\Omega|^{2/d}\ge \lambda_1(B_1)|B_1|^{2/d}$  (Faber-Krahn inequality) and H\"{o}lder inequality, one immediately checks that    
\begin{equation*}
\|\nabla w\|_{L^2}^2=\int_{\R^d}f w\,dx\le \|f\|_{L^2}\|w\|_{L^2}\le C_{d,|\Omega|}\|f\|_{L^2}\|\nabla w\|_{L^2},
\end{equation*}
where  $C_{d,|\Omega|}$ depends only on the dimension $d$ and on $|\Omega|$.  This  finally gives that 
\begin{equation}\label{eq:estH^1}
\|w_{\Omega,f}\|_{H^1}^2\le C_{d,|\Omega|} \|f\|_{L^2(\Om)}^2,
\end{equation}
where $C_{d,|\Omega|}$ is a possibly different constant, also depending only on $d$ and $|\Omega|$.

\noindent Of course, the same results hold if we replace $\widetilde H^1_0(\Omega)$ by the classical Sobolev space $H^1_0(\Omega)$ (though the function $w_{\Om,f}$ is not the same in general).\\

\noindent{\bf Other properties:}
Suppose that $\Omega\subset\R^d$ is a set of finite Lebesgue measure and suppose that $f\in L^p(\R^d)$ for some $p\in(d/2,\infty]$. Then the solution $w$ of \eqref{eqmeassets}
%$$-\Delta w=f\quad\text{in}\quad \Omega,\qquad w\in \widetilde H^1_0(\Omega),$$
has the following properties:
\begin{itemize}
\item $w$ is bounded, precisely we have (see \cite{bucve}):
\begin{equation}\label{winfty06}
\|w\|_{L^\infty}\le \frac{C_d}{2/d-1/p}\|f\|_{L^p}|\Omega|^{2/d-1/p}.
\end{equation}
In particular, if $f\equiv 1$ on $\Omega$, by letting \(p\to \infty\) we get 
\begin{equation}\label{winftyf=1}
\|w\|_{L^\infty}\le C_d|\Omega|^{2/d}.
\end{equation}
By  \cite{talenti}, we can choose \(C_d\) to be less than \(\frac{1}{2d\omega_d^{2/d}}\).

\item If $w\ge 0$, then we have the inequality (see for example \cite{BHP}) 
\begin{equation}\label{deltaw+f06}
\Delta w+f \mathbbm{1}_{\Om}\ge 0\quad\text{in sense of distributions on }\R^d.
\end{equation}
\item Thanks to \eqref{deltaw+f06}, every point $x\in\R^d$ is a Lebesgue point for $w$, i.e. $w$ has a representative defined everywhere on $\R^d$.
\end{itemize}

\subsection{The Dirichlet energy functionals}
For an open set $\Omega\subset\R^d$ of finite measure, the Dirichet energy $E_f(\Omega)$, is defined as 
$$E_f(\Omega)=\min_{u\in H^1_0(\Omega)}J_f(u).$$
Alternatively, the Dirichlet energy $\widetilde E_f(\Omega)$ is defined for every set of finite measure $\Omega\subset\R^d$ as 
$$\widetilde E_f(\Omega)=\min_{u\in \widetilde H^1_0(\Omega)}J_f(u)=J_f(w_{\Omega,f}).$$
\noindent A simple integration by parts, which  is expressed through \eqref{testwithsol} for irregular domains, gives
$$\widetilde E_f(\Omega)=-\frac12\int_{\R^d}f w_{\Omega, f}\,dx.$$
We notice that, since $J_f(0)=0$, we have that $\widetilde E_f(\Omega)\le 0$, where the inequality is strict if $f\not\equiv 0$. %Moreover, an equality is achieved if and only if $\widetilde H^1_0(\Omega)=\{0\}$. Indeed, if $0\not\equiv u\in \widetilde H^1_0(\Omega)$, then 
%$$\widetilde E_f(\Omega)\le \min_{t\in \R}\Big\{\frac{t^2}{2}\int_{\R^d}|\nabla u|^2\,dx-t\int_{\R^d}uf\,dx\Big\}<0.$$
If $\Omega_1\subset\Omega_2$, then $\widetilde H^1_0(\Omega_1)\subset \widetilde H^1_0(\Omega_2)$ and so $\widetilde E_f(\Omega_1)\ge \widetilde E_f(\Omega_2)$.  Moreover for an  open set $\Omega$,  $E_f(\Omega)\ge \widetilde E_f(\Omega)$ and there is equality if $H^1_0(\Omega)=\widetilde H^1_0(\Omega)$.

\subsection{Eigenvalues and eigenfunctions} 
We first notice that the operator $R_\Omega$, that associates to a function $f\in L^2(\R^d)$ the solution $w_{\Omega,f}$ of \eqref{eqmeassets}, is a bounded linear operator $R_\Omega: L^2(\R^d)\to L^2(\R^d)$ with norm depending only on the dimension and the measure of $\Omega$. Moreover:
\begin{itemize}
\item $R_\Omega$ is compact due to the compact inclusion $\widetilde H^1_0(\Omega)\subset L^2(\R^d)$;
\item $R_\Omega$ is self-adjoint since
$$\int_{\R^d}fR_\Omega(g)\,dx=\int_{\R^d}\nabla R_\Omega(f)\cdot\nabla R_\Omega(g)\,dx=\int_{\R^d}gR_\Omega(f)\,dx\quad\hbox{for all}\quad f,g\in L^2(\R^d);$$
\item $R_\Omega$ is positive since 
$$\int_{\R^d}fR_\Omega(f)\,dx=\int_{\R^d}|\nabla R_\Omega(f)|^2\,dx,$$
 which is strictly positive if $f\not\equiv 0$.
\end{itemize}
As a corollary of these properties, the spectrum of $R_{\Om}$ consists of a sequence of eigenvalues $\widetilde \Lambda_1(\Omega)\ge \widetilde \Lambda_2(\Omega)\ge\dots\ge \widetilde \Lambda_k(\Omega)\ge \dots> 0$ decreasing to 0. 
 We define the eigenvalues of the Dirichlet Laplacian on the measurable set $\Omega$ as   $\widetilde \lambda_k(\Omega)=\widetilde \Lambda_k(\Omega)^{-1}$ and the corresponding (normalized) eigenfunctions  $u_k\in \widetilde H^1_0(\Omega)$ as  
$$-\Delta u_k=\widetilde \lambda_k(\Omega) u_k\quad\hbox{in}\quad \Omega,\qquad u_k\in\widetilde H^1_0(\Omega),\qquad \int_{\R^d}u_k^2\,dx=1.$$ 
Note that we have the following min-max characterisation for   $\widetilde \lambda_k(\Omega)$:
$$\ds\widetilde \lambda_k(\Omega)=\min_{\substack{S_k\subset  \widetilde{H}^1_0(\Omega)\\ \text{dim } S_k=k}}\, \max_{u\in S_{k}\setminus\{0\}}\displaystyle{\frac{\int_{\R^d}|\nabla u|^2\,dx}{\int_{\R^d}u^2\,dx}},$$
where the minimum is taken over the $k$-dimensional subspaces $S_k$ of $\widetilde H^1_0(\Omega)$.
In particular  the Dirichlet eigenvalues are decreasing with respect to set inclusion, i.e. $\widetilde \lambda_k(\Omega_1)\ge \widetilde \lambda_k(\Omega_2)$ whenever $\Omega_1\subset \Omega_2$.

The construction of the Dirichlet eigenvalues and the resolvent operator in the classical case $H^1_0(\Omega)$, where $\Omega$ is an open set of finite measure, is precisely the same and again we have 
$$\lambda_k(\Omega)=\min_{\substack{S_k\subset H^1_0(\Omega)\\ \text{dim } S_k=k}}\, \max_{u\in S_{k}\setminus\{0\}}\frac{\int_{\R^d}|\nabla u|^2\,dx}{\int_{\R^d}u^2\,dx}.$$
Since $\widetilde \lambda_k$ is defined  as minimum over a larger space than $\lambda_k$, clearly $\lambda_k(\Omega)\ge \widetilde\lambda_k(\Omega)$ and  equality is achieved if $H^1_0(\Omega)=\widetilde H^1_0(\Omega)$.

\subsection{Sets of finite perimeter}\label{ssect:perimeter}
For a measurable set $\Omega\subset\R^d$, we define its perimeter by 
\begin{equation}\label{sheldon}
P(\Omega):=\sup\Big\{\int_{\Omega}\text{div} \phi\,dx\ :\ \phi\in C^1_c(\R^d;\R^d),\ |\phi|\le 1\ \text{on}\ \R^d\Big\},\end{equation}
(where $|\cdot|$ denotes the euclidian norm). It is well known that if the set $\Omega$ is regular then the above definition coincides with the usual definition of the perimeter. We say that a set has finite perimeter if $P(\Om)<\infty$ and we  refer to the books \cite{maggi}, \cite{giusti} and \cite{amfupa} for an introduction to the theory of the sets of finite perimeter. Here we recall some basic properties of these sets. If \(\Omega\) has finite perimeter then    
the distributional derivative $\nabla\ind_{\Omega}$ of the characteristic function $\ind_\Omega$ is a Radon measure. We then define the  reduced boundary $\partial^\ast\Omega$  as the set of points $x\in\R^d$ such that 
$$\text{the limit}\quad\nu_\Omega(x):=\lim_{r\to0}\frac{\nabla\ind_\Omega(B_r(x))}{|\nabla\ind_\Omega|(B_r(x))}\quad\text{exists and is such that }\quad |\nu_\Omega(x)|=1,$$ 
where $|\nabla\ind_\Omega|$ is the total variation of $\nabla\ind_\Omega$.  We recall that  $\partial^*\Omega\subset\partial\Omega$ (see also Section \ref{ssect:representative}) and that \(P(\Omega)=\mathcal H^{d-1}(\partial^*\Omega)\).

We say that the set $\Omega\subset\R^d$ is a local $\alpha$-quasi-minimizer for the perimeter in the open set $\Dr\subset\R^d$, if there are constants $C>0$ and $r_0>0$ such that, for every $r\in(0,r_0)$ and $x\in\R^d$, we have 
$$P(\Omega)\le P(\widetilde\Omega)+Cr^\alpha,\quad\text{for every measurable set}\quad\widetilde{\Omega}\subset\R^d\quad\text{such that}\quad \Omega\Delta\widetilde\Omega\subset B_r(x)\cap \Dr.$$
Our main tool to prove  regularity is the following  theorem:
\begin{teo}[Tamanini \cite{tamanini}]\label{tama}
Suppose that the set of finite measure $\Omega\subset\R^d$ is a local $\alpha$-quasi-minimizer of the perimeter in $\Dr$ for $\alpha\in(d-1,d]$. Then 
\begin{enumerate}
\item[(R1)] The set $\partial^\ast \Omega\cap \Dr$ is locally the graph of a $C^{1,\frac{\alpha-d+1}2}$ function.
\item[(R2)]  The singular set has dimension at most $d-8$, i.e. $\mathcal{H}^s(\partial\Omega\setminus\partial^\ast\Omega)=0$ for every $s<d-8$, where $\mathcal{H}^s$ is the $s$-dimensional Hausdorff measure.
\end{enumerate}
\end{teo}
In this statement, $\partial\Om$ is the topological boundary for a suitable representative of $\Om$, see \eqref{eq:representative} and \eqref{eq:topboundary}.

%{\Rd Finally, we recall the definition of the relative perimeter: if $\Om, D$ are measurable sets in $\R^d$, then
%\begin{equation}\label{sheldon}
%P(\Omega;D):=\sup\Big\{\int_{\Omega}\text{div} \phi\,dx\ :\ \phi\in C^1_c(D;\R^d),\ |\phi|\le 1\ \text{on}\ D\Big\},\end{equation}
%}
\subsection{Set representatives}\label{ssect:representative}

Typically when we speak of a domain in shape optimization, we actually mean an equivalence class of domains. When it comes to regularity of the optimal domains this may cause some problems. For example, the ball $B_1$ is a solution of the shape optimization problem 
$$\min\Big\{\lambda_1(\Omega)+P(\Omega)\ :\ \Omega\subset\R^d,\ |\Omega|=1\Big\},$$
but the set $B_1\setminus\{0\}$ is also a solution. Thus, it is natural to expect that the regularity theory will apply only to a certain representative of the optimal set. 
In this section, we make a few remarks about the choice of representative of a domain  $\Om$.

\begin{itemize}
\item When dealing with sets of finite  perimeter, it is classical to identify a measurable set $\Omega$ with its class of equivalence given by the relation $\Omega_1\sim\Omega_2$ if and only if $|\Omega_1\Delta\Omega_2|=0$ (notice that by \eqref{sheldon} the function $P$ does not depend on the representative). One can then choose a representative so that $\partial\Om$ is minimal: in \cite[Proposition 3.1]{giusti} or \cite[Proposition 12.19]{maggi} it is proved that 
%there exists a representative of $\Om$ such that
%$$\Big(|B_r(x_{0})\cap\Omega^c|=0\Big)\Rightarrow \Big(B_r(x_{0})\subset\Omega\Big)\quad\text{and}\quad \Big(|B_r(x_{0})\cap\Omega|=0\Big)\Rightarrow \Big(B_r(x_{0})\subset\Omega^c\Big),$$
\begin{multline}\label{eq:representative}\Om\sim (\Om\cup\Om_{1})\setminus\Om_{0}\textrm{ where }\Om_{1}=\{x, \exists r>0, |\Om\cap B_{r}(x)|=|B_{r}|\},\\ \;\;\;
\Om_{0}=\{x, \exists r>0, |\Om\cap B_{r}(x)|=0\}
\end{multline}
and that if we choose this representative, we have $\partial\Om=\partial^M\Om=\overline{\partial^*\Om}$ where 
\begin{equation}\label{eq:topboundary}
\partial^M\Om=\{x\in\R^d, 0<|\Om\cap B_{r}(x)|<|B_{r}|, \;\;\forall r>0\}.
\end{equation}
\item When dealing with shape functionals involving the Sobolev space $H^1_{0}(\Omega)$ (where $\Om$ is open or quasi-open), it is more suitable to identify a set with its class of equivalence given by $\Om_{1}\sim\Om_{2}$ if and only if $\cp(\Om_{1}\Delta\Om_{2})=0$, which identifies sets more accurately than in the previous item. In order to define a convenient canonical representative of a set $\Omega$, we first consider the solution $w_\Omega$  of the equation 
$$-\Delta w_{\Om}=1\quad\hbox{in}\quad\Omega,\qquad w_{\Om}\in H^1_0(\Omega).$$
It is different from $w_{\Om}$ in Section \ref{ssect:elliptic} and we will denote it $\widetilde{w}_{\Om}$ for the purpose of this section. We recall that, since $\Delta \left(w_\Omega+\frac{|x|^2}{2d}\right)=\Delta w_\Omega+1\ge 0$ (in $\Om$ and so in $\R^d$, see \eqref{deltaw+f06}), we have that every point of $\R^d$ is a Lebesgue point for $w_\Omega$ and so we can choose a canonical representative of $w_\Omega$ defined pointwise \emph{everywhere} by 
$$\ds w_\Omega(x)=\lim_{r\to 0}\frac1{|B_r|}\int_{B_r(x)}w_\Omega(y)\,dy.$$
 Therefore the set $\{w_{\Om}>0\}$ is well-defined, is a quasi-open set and we have that $H^1_0(\Omega)=H^1_0(\{w_\Omega>0\})$ (see for example \cite{hepi05} for more details). Thus, we can restrict our attention to sets of the form $\{w_{\Om}>0\}$ which, in the case of quasi-open sets $\Omega$ are representatives of $\Omega$, in the equivalence class defined above.  
%(well-defined as $w_{\Om}$ is lower-semicontinuous on $D$), where $w_{\Om}$ is defined by
%$$-\Delta w_{\Om}=1\quad\hbox{in}\quad\Omega,\qquad w_{\Om}\in H^1_0(\Omega),$$
%This set is a quasi-open set, see for example \cite{hepi05} for a definition.\\

We notice that, with the formulation \eqref{eq:P+Ef3} of our problem, one cannot expect a full regularity result for the boundary of such representative. Indeed, let us consider for example the (smooth) set $\Om^*$ solving
\begin{equation*}
\min \Big\{\lambda_{2}(\Omega),\ \Om\subset \R^2,\ P(\Om)=p\Big\},
\end{equation*}
studied in \cite{bubuhe} and which solves \eqref{eq:P+Ef3} for $\mathcal{G}=\lambda_2$ and a suitable choice of $m$. Then, any set of the form $\Om^*\setminus \Sigma$ where $\Sigma$ is any closed subset of the nodal line is again a minimizer, since its perimeter is the same as $\Om$ (as the perimeter does not see the set of zero measure) and \(\lambda_2(\Omega^*\setminus \Sigma)=\lambda_2(\Omega^*)\).\\
\item We now use the ideas from the previous two paragraphs to construct a canonical representative of an \emph{optimal} measurable set $\Omega\subset\R^d$.   Reasoning as above, we introduce the solution $\widetilde w_\Omega$ of the problem
$$-\Delta \widetilde{w}_{\Om}=1\quad\hbox{in}\quad\Omega,\qquad \widetilde{w}_{\Om}\in \widetilde{H}^1_0(\Omega),$$ 
which is defined pointwise everywhere on $\R^d$. 
Thus the set $\{\widetilde{w}_{\Om}>0\}$ is well-defined, and one has $\{\widetilde{w}_{\Om}>0\}\subset\Om$ $a.e.$  and  equality holds if and only if $\Omega$ is quasi-open, up to a set of measure zero. Moreover, for every set of finite measure $\Omega\subset\R^d$, we have
$$\widetilde{H}^1_{0}(\Om)=\widetilde{H}^1_{0}(\{\widetilde{w}_{\Om}>0\})=H^1_{0}(\{\widetilde{w}_{\Om}>0\}),$$
which gives that all the spectral functionals on $\Omega$ and $\{\widetilde w_\Omega>0\}$ have the same values.\\
We now suppose that $\Omega$ satisfies an exterior density estimate, which is the case (as we will prove in Section \ref{sect:sup}) when $\Omega$ is optimal for the functionals of the form $P+\mathcal G$. In this case, we have that (see \cite[Remark 2.3, Proposition 4.7]{deve} and  Lemma \ref{olddens06} below)
\begin{equation}\label{eq:density1}
\{\widetilde{w}_{\Om}>0\}=\Om^{(1)}:=\left\{x\in\R^d, \;\;\lim_{r\to 0}\frac{|\Om\cap B_{r}(x)|}{|B_{r}(x)|}=1\right\}
\end{equation}
which is an equality between sets, and both are a representative a.e. of $\Om$. This is a consequence of the following observations:
\begin{itemize}
\item For every measurable set $\Omega$, we have $\Omega=\Omega^{(1)}$ $a.e.$, due to the Lebesgue Theorem.
\item The exterior density estimate for $\Omega$ implies that the solution $\widetilde w_\Omega$ is H\"older continuous on $\R^d$ (again, see Lemma \ref{olddens06} for more details and references). 
\item If $\widetilde w_\Omega$ is continuous, then $\{\widetilde w_\Omega>0\}$ is open and therefore $\{\widetilde w_\Omega>0\}\subset\Omega^{(1)}$. 
\item If $x_0$ is a point of density $1$ for $\Omega$, then by the exterior density estimate, there is a ball $B_r(x_0)$ such that $|B_r(x_0)\setminus\Omega|=0$. The maximum principle applied to the solution $\widetilde w_B(x)=\frac{(r^2-|x|^2)^+}{2d}$ of the PDE
$$-\Delta \widetilde w_B=1\quad\text{on}\quad B,\qquad w_B\in \widetilde H^1_0(B_r(x_0))=H^1_0(B_r(x_0)),$$
gives that $\widetilde w_\Omega\ge \widetilde w_B>0$ on $B_r(x_0)$, which shows that $x_0\in \{\widetilde w_\Omega>0\}$ and so $ \Omega^{(1)}\subset\{\widetilde w_\Omega>0\}$.
\end{itemize}
%then we have up to capacity (and so up to Lebesgue measure, so that these sets have the same perimeter)
Finally, again by the exterior density estimates,  $\Om^{(1)}$ is equal to the representative defined in \eqref{eq:representative}, and 
the regularity result that we prove in this paper, precisely refers to these representatives (note that this  is also the case for the results stated in  Theorem \ref{tama}). Moreover, for such a representative, the classical formulation \eqref{eq:P+Ef3} and the generalized one \eqref{eq:P+Ef4} from Section \ref{sect:existence} are equivalent. This will allow us to obtain existence of an optimal set in Theorem \ref{th:mainD}, see  Section \ref{sect:final}.
\end{itemize}

\subsection{Convergence of measurable sets} Suppose that $\Omega_n\subset\R^d$ is a sequence of measurable sets of uniformly bounded Lebesgue measure $|\Omega_n|\le C$. Consider the torsion functions $w_{\Omega_n}$ solutions of the equations 
$$-\Delta w_{\Omega_n}=f\quad\text{in}\quad \Omega_n,\qquad w_{\Omega_n}\in \widetilde H^1_0(\Omega_n),$$
and suppose that the sequence $w_{\Omega_n}$ converges strongly in $L^2(\R^d)$ to a function $w\in H^1(\R^d)$. Then setting $\Omega=\{w>0\}$, one easily checks that:
\begin{itemize}
\item the Lebesgue measure is lower semicontinuous
\begin{equation*}
|\Omega|\le\liminf_{n\to\infty}|\Omega_n|;
\end{equation*} 
\item the Dirichlet eigenvalues $\widetilde\lambda_k$ are lower semicontinuous
\begin{equation}\label{lsclambda06}
\widetilde\lambda_k(\Omega)\le\liminf_{n\to\infty}\widetilde\lambda_k(\Omega_n);
\end{equation} 
\item the Dirichlet energy with respect to any $f\in L^p(\R^d)$, with $p\in[2,\infty]$, is lower semicontinuous 
\begin{equation}\label{lscef06}
\widetilde E_f(\Omega)\le\liminf_{n\to\infty}\widetilde E_f(\Omega_n).
\end{equation} 
\end{itemize}

\begin{oss}\label{rk:convL1}
Suppose that the sequence of sets of finite measure $\Omega_n$ converges in $L^1(\R^d)$ to the set $\Omega\subset\R^d$. Then  the semicontinuity properties \eqref{lsclambda06} and \eqref{lscef06} also hold (see for example \cite{hepi05}). 
\end{oss}

We notice that the family $(w_{\Omega_n})_n$ is relatively compact in $L^2$ whenever $\Omega_n\subset\Dr$ for a set of finite measure $\Dr\subset\R^d$. This is no more the case when $\Dr=\R^d$. However we can apply the concentration-compactness principle of P.L. Lions to the sequence of characteristic functions $\ind_{\Omega_n}$ and use the bound $w_{\Omega_n}\le C\ind_{\Omega_n}$ to control the behaviour of $w_{\Omega_n}$. Precisely, we have the following result, see \cite[Theorem 3.1]{deve} and  \cite{bubuhe}.
\begin{teo}\label{ccth06}
Suppose that the sequence $\Omega_n\subset\R^d$ has uniformly bounded measure and perimeter: $|\Omega_n|+P(\Omega_n)\le C$. Then, up to a subsequence, we have one of the following possibilities: 
\begin{enumerate}
\item[(1a)]\emph{Compactness.} There is a set of finite perimeter $\Omega\subset\R^d$ such that $\ind_{\Omega_n}$ converges to $\ind_{\Omega}$ in $L^1(\R^d)$.
\item[(1b)]\emph{Compactness at infinity.} There is a set of finite perimeter $\Omega\subset\R^d$ and a sequence $(x_n)_{n\ge 1}\subset\R^d$ such that the sequence $x_n+\Omega_n$ converges to $\Omega$ in $L^1(\R^d)$.
\item[(2)]\emph{Vanishing.} For every $R>0$
$$\lim_{n\to\infty}\sup_{x\in\R^d}|B_R(x)\cap\Omega_n|=0.$$
Moreover, for every $f\in L^p(\R^d)$ with $p\in (d/2,\infty]$ and every $k\ge 1$, we have 
$$\lim_{n\to\infty}\|w_{\Omega_n, f}\|_{L^\infty}=0\quad\text{and}\quad\lim_{n\to\infty}\widetilde\lambda_k(\Omega_n)=+\infty.$$
\item[(3)]\emph{Dichotomy.} There are sequences $A_n$ and $B_n$ such that 
\begin{itemize}
\item $A_{n}\cup B_n\subset\Omega_n $ and $ \ds\lim_{n\to\infty}\text{dist}(A_n,B_n)=+\infty$;
\item $\ds \lim_{n\to\infty}|\Omega_n\setminus(A_n\cup B_n)|=0 $ and $ \ds \lim_{n\to\infty}| P(\Omega_n)-P(A_n\cup B_n)|=0$;
\item $\ds \lim_{n\to\infty}\|R_{\Omega_n}-R_{A_n\cup B_n}\|_{\mathcal{L}(L^2(\R^d);L^2(\R^d))}=0$.\\
Moreover, for every $k\in\N$ and every $f\in L^p(\R^d)$ with $f\in[2,\infty]$, we have 
$$\ds \lim_{n\to\infty}|\widetilde\lambda_k(\Omega_n)-\widetilde\lambda_k(A_n\cup B_n)|=0\quad\text{and}\quad\ds \lim_{n\to\infty}|\widetilde E_f(\Omega_n)-\widetilde E_f(A_n\cup B_n)|=0.$$
\end{itemize}
\end{enumerate} 
\end{teo}

\section{Existence of optimal sets}\label{sect:existence}
In this section, we prove the following existence result. Note that we prove existence in the class of measurable sets and with $\widetilde{\mathcal{G}}$ instead of $\mathcal{G}$. Using the regularity theory developed in the following sections, we conclude in Section \ref{sect:final} to existence (and regularity) of solutions to Problem \eqref{eq:P+Ef3}.
\begin{prop}\label{prop0}
Suppose that $\Dr\subset\R^d$ is a bounded open set or the entire space $\Dr=\R^d$.
Then there is a solution of the problem 
\begin{equation}\label{eq:P+Ef4}
\min \Big\{P(\Om)+\widetilde{\mathcal{G}}(\Omega),\;\; \Om\subset D,\ |\Om|=m\Big\},
\end{equation}
%{\Rd (I dropped `open' here.)}
 where $m<|\Dr|$ and $\widetilde{\mathcal{G}}$ is one of the following functionals:
\begin{itemize}
\item $\widetilde{\mathcal{G}}=\widetilde E_f$, where $f\in L^p(D)$ with $p\in(d,\infty]$ if $\Dr$ is bounded and $p\in (d,\infty)$ if $\Dr=\R^d$;
\item $\widetilde{\mathcal{G}}=F(\widetilde\lambda_{1},\dots,\widetilde\lambda_{k})$, where $F:\R^k\to\R$ is locally Lipschitz continuous and increasing in each variable.
\end{itemize}
\end{prop}

\begin{proof}[Proof of Proposition \ref{prop0} in the case $\Dr$ bounded] 
There exist $E\subset\Dr$ a smooth set of measure $m$, and a minimizing sequence $\Omega_n\subset\Dr$ such that 
$$P(\Omega_n)+\widetilde{\mathcal{G}}(\Omega_n)\le P(E)+\widetilde{\mathcal{G}}(E).$$
By the monotonicity of $\widetilde{\mathcal{G}}$, we have that 
$$P(\Omega_n)\le P(E)+\widetilde{\mathcal{G}}(E)-\widetilde{\mathcal{G}}(\Dr),$$
i.e. the sequence $\Omega_n$ has uniformly bounded perimeter. Then there is a set of finite perimeter $\Omega\subset\Dr$ such that, up to a subsequence, we have that $|\Omega\Delta\Omega_n|\to 0$. By the lower semicontinuity of the perimeter and of  $\widetilde{\mathcal{G}}$ with respect to the $L^1$ convergence (Remark \ref{rk:convL1}), we have that 
$$P(\Omega)+\widetilde{\mathcal{G}}(\Omega)\le\liminf_{n\to+\infty} \big(P(\Omega_n)+\widetilde{\mathcal{G}}(\Omega_n)\big),$$
which proves that $\Omega$ is a solution of \eqref{eq:P+Ef4}.
\end{proof}

\begin{proof}[Proof of Proposition \ref{prop0} in the case $\Dr=\R^d$ and $\widetilde{\mathcal{G}}=\widetilde E_f$] 
In this case, the direct method does not work straightforwardly due to the fact that the boundedness of the perimeter does not imply  compactness in $L^1$. Thus we will apply the concentration compactness principle of Theorem \ref{ccth06}. Let $\Omega_n\subset\R^d$ be a minimizing sequence. As in the case of $\Dr$ bounded, we have that the perimeter is uniformly bounded $P(\Omega_n)\le C$ for some $C>0$. 
Indeed, denoting by \(w_n\) the solution of \(-\Delta w_n=f\) in \(\widetilde H^1_0(\Omega_n)\), we have, according to \eqref{winfty06}, that
\[
\widetilde E_f(\Omega_n)=-\frac{1}{2}\int w_n f\ge -\|w_n\|_{L^{p'}}\|f\|_{L^p}\ge -\|w\|_{\infty} |\Omega_n|^{1/p'}\|f\|_{L^p}\ge -C(m,\|f\|_{L^p}).
\]
Hence, by taking any smooth set \(E\) with measure \(m\), we infer
\[
P(\Omega_n)\le C(m,\|f\|_{L^p})+P(E)+\widetilde E_f(E).
\]

We now have three possibilities:
\begin{itemize}
\item \emph{Compactness.} Suppose that $\ind_{\Omega_n}$ converges strongly in $L^1(\R^d)$ to $\ind_{\Omega}$ for some $\Omega\subset\R^d$. Then $\Omega$ solves \eqref{eq:P+Ef4} by the semicontinuity of $P$ and $\widetilde E_f$. 
\item \emph{Compactness at infinity.}  If $f$ is not constantly zero (the case $f=0$ being trivial), there cannot be a divergent sequence $x_n$ and a set $\Omega$ such that $x_n+\Omega_n$ converges in $L^1$ to $\Omega$. Indeed, if it was the case, then we would get that, up to a subsequence, $w_{\Omega_n,f}(\cdot+x_n)$ converges in $L^{p'}(\R^d)$ to some $w\in H^1_0(\Omega)$. In particular, we would have 
$$\widetilde E_f(\Omega_n)=-\frac12\int_{\R^d}w_{\Omega_n,f}f\,dx=-\frac12\int_{\R^d}w_{\Omega_n,f}(x_n+x)f(x_n+x)\,dx\longrightarrow 0,$$
since  $f(x_n+\cdot)\rightharpoonup 0$ weakly in $L^p(\R^d)$. Thus, we would get that 
$$\liminf_{n\to\infty}P(\Omega_n)+\widetilde E_f(\Omega_n)=\liminf_{n\to\infty} P(\Omega_n)\ge P(B),$$
where $B$ is a ball of measure $m$. If $f$ is not constantly zero, this is a contradiction with the fact that $\Omega_n$ is minimizing since the total energy $P(B)+\widetilde E_f(B)$ of the ball $B$ is strictly smaller than $P(B)$ each time when we choose $B$ such that $f$ is not constantly vanishing in $B$.
\item \emph{Vanishing.} The vanishing cannot occur for a minimizing sequence $\Omega_n$. Indeed, if $\Omega_n$ was a vanishing sequence, then we would have that $w_{\Omega_n,f}$ converges to zero in $L^{\infty}(\R^d)$ and so in $L^{p'}(\R^d)$. Thus also the energy converges to zero, that is $\ds \widetilde E_f(\Omega_n)=-\frac12\int_{\R^d}w_{\Omega_n,f}f\,dx\to 0$, which is a contradiction with the fact that $\Omega_n$ is a minimizing sequence, by the same argument as in the previous case. 
\item \emph{Dichotomy.} If the dichotomy occurs, then there is a sequence $\Omega_n'\subset\Omega_n$ such that 
\begin{itemize}
\item $|\Omega_n\setminus \Omega_n'|\to 0$;
\item $\Omega_n'=A_n\cup B_n$, $\text{dist}(A_n, B_n)\to+\infty$, $\ds\lim_{n\to\infty} |A_n|=m_1>0$ and $\lim_{n\to\infty}|B_n|=m_2>0$;
\item $\ds\lim_{n\to \infty} \big(P(\Omega_n)+\widetilde{\mathcal G}(\Omega_n)\big)=\lim_{n\to \infty} \big(P(\Omega_n')+\widetilde{\mathcal G}(\Omega_n')\big)$.
\end{itemize}   
Now we notice that
$$\widetilde{E}_f(\Omega_n')=\widetilde E_f(A_n)+\widetilde E_f(B_n).$$
On the other hand, since $\text{dist}(A_n, B_n)\to+\infty$ and $f\in L^p(\R^d)$ with $p<\infty$, we have 
$$\text{either}\qquad \int_{A_n}f(x)\,dx\to 0\qquad\text{or}\qquad \int_{B_n}f(x)\,dx\to 0.$$
Assume without loss of generality  that $\int_{B_n}f(x)\,dx\to 0$. Since the solution $w_n$ of 
$$-\Delta w_n=f\quad\text{in}\quad B_n,\qquad w_n\in \widetilde H^1_0(B_n),$$
is bounded by a constant $\|w_n\|_{L^\infty}\le C$ that does not depend on $n$ (see \eqref{winfty06}),  we have 
$$\widetilde E_f(B_n)=-\frac12\int_{B_n}w_n f\,dx\to 0,$$
which implies that 
$$\ds\lim_{n\to \infty} \big(P(\Omega_n)+\widetilde{\mathcal G}(\Omega_n)\big)=\lim_{n\to \infty} \big(P(A_n)+\widetilde{\mathcal G}(A_n)+P(B_n)\big).$$
We now apply the concentration compactness principle to the sequence $A_n$ which will give us three more possibilities. 
\begin{itemize}
\item \emph{Compactness of $A_n$.} In this case, there exists a set of finite perimeter $A$ such that 
$$P(A)+\widetilde{\mathcal{G}}(A)\le \liminf_{n\to \infty} \big(P(A_n)+\widetilde{\mathcal G}(A_n)\big).$$
We notice that $A$ solves the problem 
\begin{equation*}
\min \Big\{P(\Om)+\widetilde{\mathcal{G}}(\Omega)\ :\  \Omega\subset\R^d,\ |\Om|=m_1\Big\}.
\end{equation*}
Now, by Proposition \ref{prop:boundedness}, the set $A$ is bounded. Then, taking any ball of measure $m_2$ disjoint with $A$, we have that $A\cup B$ is such that
\[
\begin{split}
P(A\cup B)+\widetilde{\mathcal{G}}(A\cup B)&\le P(A)+\widetilde{\mathcal{G}}(A)+P(B)\\
&\le \liminf_{n\to \infty} \big(P(A_n)+\widetilde{\mathcal G}(A_n)+P(B_n)\big),
\end{split}
\]
which proves that $A\cup B$ solves \eqref{eq:P+Ef4}.
\item \emph{Compactness at infinity of $A_n$.} This case is ruled out by the same argument as for the analogous case for $\Omega_n$.
\item \emph{Vanishing of $A_n$.} The vanishing also cannot occur since again this would imply that the Dirichlet energy converges to zero which would be a contradiction with the minimizing property of $A_n\cup B_n$.
\item \emph{Dichotomy of $A_n$.} Suppose that $A_n=C_n\cup D_n$ where $C_n$ and $D_n$ are disjoint sets such that $\text{dist}(C_n, D_n)\to+\infty$. Reasoning as above, without loss of generality,  we can assume  that $\widetilde E_f(D_n)\to 0$. We now conclude that 
\begin{align*}
\ds\lim_{n\to \infty} \big(P(\Omega_n)+\widetilde{\mathcal G}(\Omega_n)\big)&=\lim_{n\to \infty} \big(P(C_n)+\widetilde{\mathcal G}(C_n)+P(D_n)+P(B_n)\big)\\
& \geq\lim_{n\to \infty} \big(P(C_n)+\widetilde{\mathcal G}(C_n)+P(D_n^\ast)+P(B_n^\ast)\big),
\end{align*}
where $B_n^\ast$ and $D_n^\ast$ are two disjoint balls of measures $|B_n|$ and $|D_n|$ respectively which are placed far away from $C_n$. We now consider a sequence of balls $E_n$ such that $|E_n|=|B_n|+|D_n|$ and that are disjoint with $C_n$. We now notice that, by the isoperimetric inequality,  there exists a positive constant $\delta>0$ such that  $P( D_n^*)+P( B_n^*)\geq \delta +P(E_n)$. Hence
\begin{align*}
\ds\lim_{n\to \infty} \big(P(\Omega_n)+\widetilde{\mathcal G}(\Omega_n)\big)\ge \delta + \lim_{n\to \infty} \big(P(C_n)+\widetilde{\mathcal G}(C_n)+P(E_n)\big),
\end{align*}
which finally gives that $\Omega_n$ cannot be a minimizing sequence, and this is a contradiction.
\end{itemize}
\end{itemize}
\end{proof}

\begin{proof}[Proof of Proposition \ref{prop0} in the case $\Dr=\R^d$ and $\widetilde{\mathcal{G}}=F(\widetilde \lambda_1,\dots,\widetilde\lambda_k)$]
We argue as in \cite{deve} by induction on $k$ using the a priori boundedness result of Proposition \ref{prop:boundedness}. 
 First note that since \(F(\widetilde \lambda_1,\dots,\widetilde\lambda_k)\ge F(0,\dots,0)\), \(P(\Omega_n)\) is uniformly bounded. In the case $k=1$, by the Faber-Krahn and the isoperimetric inequalities, we have that a ball of measure $m$ is an optimal set. Suppose that the claim is true for $i=1,\dots,k$ and consider a functional of the form 
$$\widetilde{\mathcal{G}}(\Omega)=F(\widetilde \lambda_1(\Omega),\dots,\widetilde \lambda_{k+1}(\Omega)).$$
Since the functional is invariant under translation, for a minimizing sequence $\Omega_n$, we have only three possibilities. 
\begin{itemize}
\item \emph{Compactness.} If $\Omega_n$ converges in $L^1(\R^d)$ to $\Omega$, then by semicontinuity of the perimeter and of the functional $\widetilde{\mathcal{G}}$, we have that $\Omega$ is a minimizer of \eqref{eq:P+Ef4}.
\item \emph{Vanishing.} The vanishing cannot occur since, otherwise, we would have that $\ds\lim_{n\to\infty}\widetilde\lambda_1(\Omega_n)=+\infty$ in contradiction with the minimality of the sequence $\Omega_n$, as the ball of volume $m$ has a lower energy.
\item \emph{Dichotomy.} If the dichotomy occurs, then we can replace each of the sets $\Omega_n$ by a disjoint union $A_n\cup B_n$. Then we argue by induction as in \cite{deve}, replacing each of the sets $A_n$ and $B_n$ with the optimal sets corresponding to a functional involving less eigenvalues for which we know, by the inductive step, that a minimum exists and that it is necessarily bounded by Proposition \ref{prop:boundedness}.
\end{itemize}

\end{proof}

\section{Penalization}\label{sect:pen}

In this section, we prove that we can penalize the volume constraint for minima of the problem 
\begin{equation*}
\min\Big\{P(\Omega)+{\widetilde{\mathcal{G}}}(\Omega)\ :\ \Omega\subset\Dr,\ |\Omega|=m\Big\}.
\end{equation*}
In the following proposition, we consider $\widetilde{\mathcal{G}}:\mathcal{B}(D)\to\R$ to be one of the following functionals:
\begin{itemize}
\item ${\widetilde{\mathcal{G}}}(\Omega)=\widetilde E_f(\Omega)$, for $f\in L^p(\Dr)$ with $p\in[2,\infty]$.
\item ${\widetilde{\mathcal{G}}}(\Omega)=F\big(\widetilde\lambda_1(\Omega),\dots,\widetilde\lambda_k(\Omega)\big)$, where the function $F:\R^k\to\R$ is locally Lipschitz continuous. 
\end{itemize}
We notice that we do not suppose the monotonicity of $F$, but we will assume that an optimal set exists.
% and increasing in each variable. INUTILE ici ?}
\begin{prop}\label{prop1}
Let $\Dr\subset\R^d$ be an open set and let $\Omega^*\subset\Dr$ be a minimizer of
$$\min\Big\{P(\Omega)+{\widetilde{\mathcal{G}}}(\Omega)\ :\ \Omega\subset \Dr, \ |\Omega|=m\Big\},$$
where $m<|\Dr|$ is fixed.  
%where $\G$ satisfies \eqref{eq:H2}
Then there are constants $r>0$ and $\mu<+\infty$ such that 
\begin{multline*}
P(\Omega^*)+{\widetilde{\mathcal{G}}}(\Omega^*)\leq P(\Omega)+{\widetilde{\mathcal{G}}}(\Omega)+\mu\big||\Omega|-|\Omega^*|\big|, \\\quad\textrm{for every } \Omega\textrm{ such that }\exists x\in D, \quad\Omega\Delta\Omega^*\subset B_{r}(x)\cap \Dr.
\end{multline*}
\end{prop}

We will carry out the proof of this proposition in three steps. In Subsection \ref{seclip}, we prove our main estimates involving the Dirichlet energy and the Dirichlet eigenvalues. Subsection \ref{subsecgenlip} is dedicated to a general result concerning the possibility of penalizing the volume constraint, and in Subsection \ref{subseclagproof}, we conclude the proof of the above proposition.

\subsection{Lipschitz estimates of the variations of the Dirichlet energy and of the Dirichlet eigenvalues}\label{seclip}
~\\
In this subsection, we estimate the variation of the Dirichlet energy (Lemma \ref{lipen06}) and of the Dirichlet eigenvalues (Lemma \ref{liplam06}) with respect to perturbations induced by a smooth map $\Phi:\R^d\to\R^d$ close to the identity for the $C^1$-norm : $\|\Phi\|_{1,\infty}=\sup_{x\in\R^d}|\Phi(x)|+\sup_{x\in\R^d}\|D\Phi(x)\|$.

{\it These results are also valid for 
$\widetilde{E}_{f}$ and $\widetilde{\lambda}_{k}$, the proofs being exactly similar, replacing $H^1_{0}(\Om)$ with $\widetilde{H}^1_{0}(\Om)$.}
%{\Rd (Shouldn't this Section be written with $\widetilde{E_{f}}$ instead of $E_{f}$? Idem for $\lambda_{k}$.)}\\

\begin{lemma}\label{lipen06}
Let $\Omega\subset\R^d$ be a set of finite measure, $f\in L^p(\R^d)$ a given function with $p\in[2,\infty]$, and $\Phi\in C^\infty_c(\R^d;\R^d)$ such that $\|D\Phi-Id\|_{L^\infty}\le 1/2$. Then we have the estimate 
\begin{equation*}
\big| E_f(\Phi(\Omega))-E_f(\Omega)\big|\le C_{d,|\Omega|}\|f\|_{L^p}^2\|\Phi-Id\|_{1,\infty},
\end{equation*}
where $C_{d,|\Omega|}$ is a constant depending only on the dimension $d$ and on $|\Omega|$.
\end{lemma}
%{\Rd What if $f\in L^p$ for $p\neq 2$? We annonce at the beginning of the Section $f\in L^p$ for some $p\geq 2$.}
\begin{proof}
Let $u\in H^1_0(\Omega)$ be the solution of the problem 
$$-\Delta u=f\quad\text{in}\quad\Omega,\qquad u\in H^1_0(\Omega).$$ 
On the set $\Phi(\Omega)$, we consider the test function $u\circ\Phi^{-1}\in H^1_0(\Phi(\Omega))$. Then we have 
\begin{align*}
E_f(\Phi(\Omega))-E_f(\Omega)&\le J_f(u\circ\Phi^{-1})-J_f(u)\\
&=\frac12\int_{\R^d}|\nabla (u\circ\Phi^{-1})|^2\,dx-\int_{\R^d}f u\circ \Phi^{-1}\,dx-\frac12\int_{\R^d}|\nabla u|^2\,dx+\int_{\R^d}fu\,dx\\
&\le \frac12\int_{\R^d}|\nabla u|^2\circ\Phi^{-1}\|D(\Phi^{-1})\|^2\,dx-\frac12\int_{\R^d}|\nabla u|^2\,dx-\int_{\R^d}f \big(u\circ \Phi^{-1}-u\big)\,dx\\
&=\frac12\int_{\R^d}|\nabla u|^2 \big(\|D\Phi\|^{-2} |\det D\Phi|-1\big)\,dx-\int_{\R^d}f \big(u\circ \Phi^{-1}-u\big)\,dx.
\end{align*}
We now notice that, since $u\in H^1(\R^d)$, we have (for some constant $C_{d}$ depending on the dimension)%{\Rd (it feels natural, though not completely trivial; do we have a reference?)}
\begin{equation*}
\int_{\R^d}\big|u\circ \Phi^{-1}-u\big|^2\,dx\le C_{d} \|\Phi^{-1}-Id\|_{L^\infty}^2 \int_{\R^d}|\nabla u|^2\,dx.
\end{equation*}
%and so by the Cauchy-Schwartz inequality and the identity {(\Rd see Section \ref{ssect:elliptic}: I would put the FB argument in Section \ref{ssect:elliptic}, and refer to it here)}
%\begin{equation}\label{eq:est}
%\|\nabla u\|^2_{L^2}=\int_{\R^d}uf\,dx\le \|u\|_{L^2}\|f\|_{L^2}\le \lambda_1(\Omega)^{-1/2}\|\nabla u\|_{L^2}\|f\|_{L^2},
%\end{equation}
%we get the estimate of the second term
Therefore, to analyze the second term above, we use \eqref{eq:estH^1} and the elementary inequality
$\|\Phi^{-1}-Id\|_{\infty}\leq\frac{\|\Phi-Id\|_{\infty}}{1-\|\Phi-Id\|_{\infty}}\leq 2\|\Phi-Id\|_{\infty}$ (recall that \(\|\Phi-Id\|_{\infty}\le 1/2\)) to  obtain
\begin{equation*}\label{termf06}
\Big|\int_{\R^d}f\big(u\circ \Phi^{-1}-u\big)\,dx\Big|\le C_{d,|\Om|}\|f\|_{L^p}\|\Phi^{-1}-Id\|_{L^\infty}.
\end{equation*}
In order to estimate the first term, we use that there exists $C_{d}$, a constant depending only on the dimension, such that (we recall that $\|D\Phi-Id\|_{\infty}\leq 1/2$)
$$\left|\|D\Phi\|^{-2} |\det D\Phi|-1\right|\leq C_{d}\|D\Phi-Id\|_{L^\infty}.$$
% following easy inequalities {(\Rd I say easy, but couldn't really prove so easily the first one)}
%\begin{equation*}
%\big||\det D\Phi|-1\big|\le %\big|\det (D\Phi-Id)\big|\le
% \|D\Phi\|_{L^\infty}^{d-1}\|D\Phi-Id\|_{L^\infty},
%\end{equation*}
%\begin{equation*}
%\big||D\Phi|^{-2}-1\big| \le \frac{1+\|D\Phi\|_{L^\infty}}{\|D\Phi\|_{L^\infty}^2}\|D\Phi-Id\|_{L^\infty}.
%\end{equation*}
%Thus we have 
%\begin{equation}\label{primoterm06}
%\big|\|D(\Phi^{-1})\|^2\circ\Phi |\det D\Phi|-1\big|=\big|\|D\Phi\|^{-2} |\det D\Phi|-1\big|\le C_d\|D\Phi-Id\|_{L^\infty},
%&\le \Big(\|D\Phi\|_{L^\infty}^{d-1}+\frac{1+\|D\Phi\|_{L^\infty}}{\|D\Phi\|_{L^\infty}^2}+\frac{1+\|D\Phi\|_{L^\infty}}{\|D\Phi\|_{L^\infty}^{3-d}}\|D\Phi-Id\|_{L^\infty}\Big)\|D\Phi-Id\|_{L^\infty}\\
%&\le C_d\|D\Phi-Id\|_{L^\infty},
%\end{equation}
%where the last inequality follows from the choice $\|D\Phi-Id\|_{L^\infty}\le 1/2$. Putting together the estimates \eqref{termf06} and \eqref{primoterm06} 
%we get {\Rd (something doesn't fit here, where is the term $\|\Phi^{-1}-Id\|_{\infty}$ gone?)}
We finally get
\begin{equation*}
E_f(\Phi(\Omega))-E_f(\Omega)\le C_{d,|\Om|}\|f\|_{L^p}^2\|\Phi-Id\|_{1,\infty}.
\end{equation*}
%which by the Faber-Krahn inequality $\lambda_1(\Omega)|\Omega|^{2/d}\ge \lambda_1(B_1)|B_1|^{2/d}$ gives
%\begin{equation*}
%E_f(\Phi(\Omega))-E_f(\Omega)\le C_{d,|\Omega|}\|f\|_{L^2}^2\|D\Phi-Id\|_{L^\infty},
%\end{equation*}
Repeating now the same argument with the sets $\Phi(\Omega)$, $\Phi^{-1}(\Phi(\Omega))=\Omega$, and the function $\Phi^{-1}$, we obtain the claim.% {\Rd (Ok, mais il me semble qu'on utilise $\|A^{-1}-Id\|\leq C\|A-Id\|$?)}
\end{proof}

\begin{lemma}\label{liplam06}
Let $\Omega\subset\R^d$ be a set of finite measure and $\Phi\in C^\infty_c(\R^d;\R^d)$ %invertible an 
such that $\|D\Phi-Id\|_{L^\infty}\le 1/2$. Then we have the estimate 
\begin{equation*}
\big| \lambda_k(\Phi(\Omega))-\lambda_k(\Omega)\big|\le C_{d,|\Omega|}\|\Phi-Id\|_{1,\infty},
\end{equation*}
where $C_{d,|\Omega|}$ is a constant depending only on the dimension $d$ and the measure of $\Omega$.
\end{lemma}
\begin{proof}
The proof of this result is a direct consequence of Lemma \ref{lipen06} and of an estimate involving the projection on the space of the first $k$ eigenfunctions, that can be found in \cite{bulbk}, and that we briefly reproduce  here. % {\Rd Do we mean that we repeat the argument here?}
 Suppose that $\lambda_k(\Phi(\Omega))\ge \lambda_k(\Omega)$. As in the case of the energy, we are going to estimate the difference $\lambda_k(\Phi(\Omega))-\lambda_k(\Omega)$. Let $u_1, \dots. u_k$ be the first $k$ normalized eigenfunctions on $\Omega$. Let $R_\Omega:L^2(\R^d)\to L^2(\R^d)$ and $R_{\Phi(\Omega)}:L^2(\R^d)\to L^2(\R^d)$ be the resolvent operators on $\Omega$ and $\Phi(\Omega)$. Let $P_k:L^2(\R^d)\to L^2(\R^d)$ be the projection on the subspace $V\subset H^1_0(\Omega)$ generated by the first $k$ eigenfunctions
$$P_k(u)=\sum_{j=1}^k \left(\int_{\R^d}uu_j\,dx\right) u_j.$$
Consider the operators $T_\Omega=P_k\circ R_\Omega\circ P_k$ and $T_{\Phi(\Omega)}=P_k\circ R_{\Phi(\Omega)}\circ P_k$ on the finite dimensional space $V$. 
It is immediate to check that $u_1,\dots,u_k$ and $\lambda_1(\Omega)^{-1},\dots,\lambda_k(\Omega)^{-1}$ are the eigenfunctions and the corresponding eigenvalues of $T_\Omega$. On the other hand, if we denote by $\Lambda_1,\dots,\Lambda_k$ the eigenvalues of $T_{\Phi(\Omega)}\in \mathcal{L}(V)$, we have the inequality $\Lambda_k\le\lambda_k({\Phi(\Omega)})^{-1}$.
Indeed, we have by the min-max Theorem 
\begin{align*}
\Lambda_k&=\min_{W\subset V}\ \max_{u\in V, u\perp W}\frac{\langle P_k\circ R_{\Phi(\Omega)}\circ P_k (u),u\rangle_{L^2}}{\|u\|_{L^2}^2}\\
&=\min_{W\subset L^2}\ \max_{u\in V, u\perp W}\frac{\langle R_{\Phi(\Omega)} (u),u\rangle_{L^2}}{\|u\|_{L^2}^2}\\
&\le\min_{W\subset L^2}\ \max_{u\in L^2, u\perp  W}\frac{\langle  R_{\Phi(\Omega)}(u),u\rangle_{L^2}}{\|u\|_{L^2}^2}=\lambda_k({\Phi(\Omega)})^{-1},
\end{align*}
where the minima are over the $k$-dimensional spaces $W\subset L^2$.
Thus, we have the estimate 
\begin{equation*}\label{dorinlblemmue1}
0\le\lambda_k(\Omega)^{-1}-\lambda_k({\Phi(\Omega)})^{-1}\le \lambda_k(\Omega)^{-1}-\Lambda_k\le\|T_\Omega-T_{\Phi(\Omega)}\|_{\mathcal{L}(V)},
\end{equation*}
and on the other hand
\begin{equation*}\label{suptmutnu}
\begin{array}{ll}
\ds\|T_\Omega-T_{\Phi(\Omega)}\|_{\mathcal{L}(V)}&\ds=\sup_{u\in V}\frac{\langle (T_\Omega-T_{\Phi(\Omega)})u,u\rangle_{L^2}}{\|u\|_{L^2}^2}=\sup_{u\in V}\frac{\langle (R_\Omega-R_{\Phi(\Omega)})u,u\rangle_{L^2}}{\|u\|_{L^2}^2}\\
\\
&\ds=\sup_{u\in V}\frac{1}{\|u\|_{L^2}^2}\int_{\R^d}\big(R_\Omega(u)-R_{\Phi(\Omega)}(u)\big)u\,dx\\
\\
&\ds=\sup_{u\in V}\frac{2}{\|u\|_{L^2}^2}\big|E_u(\Phi(\Omega))-E_u(\Omega)\big|,
\end{array}
\end{equation*}
which, together with Lemma \ref{lipen06}, gives the claim. The case $\lambda_k(\Phi(\Omega))\le \lambda_k(\Omega)$ is analogous and follows by the same argument applied to the set $\Phi(\Omega)$ and the function $\Phi^{-1}$.
\end{proof}

\begin{oss}
%Hypothesis \eqref{eq:H2} is satisfied by our main examples:
We notice that similar estimates have already appeared in the literature. We refer for example to the recent article \cite{BL}, where it is proven that there exists $C$ (independent on $\Om$) such that, for any (open) set $\Om$, we have
$$|\lambda_{k}(\Phi(\Om))-\lambda_{k}(\Om)|\leq C\lambda_{k}(\Om)\|\Phi-Id\|_{1,\infty}, \;\;\;\;\textrm{ if }\|\Phi-Id\|_{1,\infty}\leq \frac{1}{C}.$$
\end{oss}

\subsection{A general result on penalization}\label{subsecgenlip}

In this subsection, we prove a lemma identifying a general set of hypotheses implying the possibility to (locally) penalize the volume constraint.

Let $\Dr\subset\R^d$ be a given open set. In the following lemma, we will denote by $\mathcal{A}$ the class of open, quasi-open, or measurable subsets of $\Dr$. For a set $\Omega\in\mathcal{A}$ and a positive real number $r>0$, we will denote by $\mathcal{A}(\Omega,r)$ the family of local perturbations of $\Omega$, i.e. 
\begin{equation*}
\mathcal{A}(\Omega,r)=\Big\{\widetilde\Omega\in\mathcal{A}\ :\ \exists x\in\R^d\text{ such that } \widetilde\Omega\Delta\Omega\subset B_r(x)\Big\}.
\end{equation*}
\begin{lemma}\label{genlam06}
Let $\Omega^\ast\in\mathcal{A}$ be a solution of the problem 
$$\min\Big\{\mathcal{F}(\Om)\ :\  \Om\in\mathcal{A},\ |\Om|=m\Big\},$$
where $m<|\Dr|$ and $\mathcal{F}:\mathcal{A}\to\R$ is a given functional. Suppose that $\Omega^\ast$ and $\mathcal{F}$ satisfy the following condition:
%\begin{equation}\label{genlam06cond}
%\left.\begin{array}{l}
%\textrm{ There are constants }\rho>0,\ \eps>0\textrm{ and }C>0\\[2mm]
%\textrm{ such that  for every }\Om\in\mathcal{A}(\Omega^\ast,\rho)\textrm{ such that }\mathcal{F}(\Omega)\le \mathcal{F}(\Omega^\ast),\\[2mm]
%\textrm{ and every }\Phi\in C^\infty_{c}(\Dr,\R^d)\\[2mm]
%\textrm{ such that }\|\Phi-Id\|_{1,\infty}<\eps\textrm{ and }\Phi=Id\textrm{ on }\Omega\Delta\Omega^\ast,\end{array}\right\}
%\textrm{ we have }\mathcal{F}(\Phi(\Om))-\mathcal{F}(\Om)\leq C\|\Phi-Id\|_{1,\infty}.
%\end{equation}

\begin{equation}\label{genlam06cond}\begin{array}{ll}
\exists (\rho,\eps,C)\in(0,\infty)^3, &\forall\; \Om\in\mathcal{A}(\Omega^\ast,\rho)\textrm{ s.t. }\mathcal{F}(\Omega)\le \mathcal{F}(\Omega^\ast),\\[1mm]
&\forall\;\Phi\in C^\infty_{c}(\Dr,\R^d)
\textrm{ s.t. }\|\Phi-Id\|_{1,\infty}<\eps\textrm{ and }\Phi=Id\textrm{ on }\Omega\Delta\Omega^\ast,\\[2.5mm]
&\;\;\;\;\;\;\hspace{3cm}\textrm{ we have }\mathcal{F}(\Phi(\Om))-\mathcal{F}(\Om)\leq C\|\Phi-Id\|_{1,\infty}.
\end{array}
\end{equation}
Then there exist $\mu\ge 0$ and $r_{0}\in(0,\rho]$ such that $\Omega^\ast$ is a solution of the problem 
$$\min\Big\{\mathcal{F}(\Om)+\mu\big||\Omega|-m\big|\ :\  \Om\in\mathcal{A}(\Omega^\ast,r_{0})\Big\}.$$
\end{lemma}
\begin{proof}
Consider two distinct points {$x_1,x_2\in\partial^M\Omega^\ast\cap\Dr$} and a number $r_1\in (0,\rho]$ sufficiently small such that 
$$B_{r_1}(x_1)\subset\Dr,\quad B_{r_1}(x_2)\subset\Dr,\quad\text{and}\quad r_1<|x_1-x_2|/4.$$
We consider two vector fields $T_1\in C^\infty_c(B_{r_1}(x_1);\R^d)$ and $T_2\in C^\infty_c(B_{r_1}(x_2);\R^d)$ such that
$$\int_{\Omega^*} \text{div}\,T_1\,dx>0\quad\hbox{and}\quad \int_{\Omega^*} \text{div}\,T_2\,dx>0.\footnote{We notice that such vector fields exist. Indeed if $\int_\Omega \text{div}\,T\,dx=0$ for every vector field $T\in C^\infty_c(B_{r_{1}};\R^d)$, then $D\mathbbm1_\Omega=0$ in \(B_{r_1(x_1)}\) in the distributional sense and thus either \( |\Omega\cap B_{r_{1}}(x_1)|=0\) or \( |\Omega\cap B_{r_{1}}(x_1)|=|B_{r_{1}}(x_1)|\) in contradiction with \(x_1\in \partial^M\Omega\).
% In this case we can approximate $\Omega$ by smooth sets and take $T$ to be an appropriate extension of the normal vector field on the boundary of one of the approximating sets.
  We refer to \cite[Section II.6]{maggi} for more details.}$$
Let $t_0>0$ satisfy the inequality $t_0\le \eps\Big(\max\big\{\|\nabla T_1\|_{L^\infty}, \|\nabla T_2\|_{L^\infty}\big\}\Big)^{-1}$ and be such that the functionals 
$$\Phi^1_t=Id+tT_1\quad\hbox{and}\quad \Phi^2_t=Id+tT_2,$$
are diffeomorphisms respectively of $B_{r_1}(x_1)$ and $B_{r_1}(x_2)$, for every $t\in(-t_0,t_0)$. We now notice that for $i=1,2$, we have the asymptotic expansion (see \cite[Theorem II.6.20]{maggi})
\begin{equation*}\label{meas06}
|\Phi^i_t(\Omega^*\cap B_{r_1}(x_i))|=|\Omega^*\cap B_{r_1}(x_i)|+t\int_{\Omega^*}\text{div}\,T_i\,dx+O(t^2).
\end{equation*}
Thus, for $t_0$ small enough, there is a constant $C_0$ depending on $T_1$ and $T_2$ such that 
$$t\le C_0\Big| |\Phi^i_t(\Omega^*\cap B_{r_1}(x_i))|-|\Omega^*\cap B_{r_1}(x_i)|\Big|,\quad \forall t\in(-t_0,t_0),\ i=1,2.$$
Now let $B_{r_{0}}(x)\subset\R^d$ be an arbitrary ball of radius $r_{0}= \min\{r_1, r_2\}$, where $r_2$ is such that $|B_{r_2}|=t_0/C_0$. Let $\Omega\subset\Dr$ be such that $\Omega\Delta\Omega^\ast\subset B_{r_{0}}(x)$. We notice that $B_{r_{0}}(x)$ does not intersect at least one of the balls $B_{r_1}(x_1)$ and $B_{r_1}(x_2)$. Without loss of generality, we suppose that $B_{r_{0}}(x)\cap B_{r_1}(x_1)=\emptyset$. Consider the set  
$\widetilde\Omega=\Phi_t^1(\Omega)$, where $t$ is such that $|\widetilde \Omega|=|\Omega^\ast|$.\footnote{Notice that the existence of such a $t$ is guaranteed by the choice $r_{0}\le r_2$.}
By the optimality of $\Omega^\ast$, we conclude that 
\begin{align*}
\mathcal{F}(\Omega^\ast)&\le \mathcal{F}(\widetilde\Omega)\le \mathcal{F}(\Omega)+C\|\nabla T_1\|_{L^\infty} C_0 \big||\widetilde\Omega|-|\Omega|\big|= \mathcal{F}(\Omega)+\mu\big||\Omega|-m\big|,
\end{align*}
where we set $\mu=C\|\nabla T_1\|_{L^\infty} C_0$.
\end{proof}

Using this general result and the estimates from Subsection \ref{seclip}, we are in position to prove Proposition \ref{prop1}.

\subsection{Proof of Proposition \ref{prop1}}\label{subseclagproof}
In view of Lemma \ref{genlam06}, it is sufficient to check that the functionals $P+\widetilde E_f$ and $P+F(\widetilde\lambda_1,\dots,\widetilde\lambda_k)$ satisfy the condition \eqref{genlam06cond}. 
\begin{itemize}
\item
For the perimeter, we use the area formula (see \cite[Proposition II.6.1]{maggi})  
$$P(\Phi(\Omega^*))=\int_{\partial^\ast\Omega^*}|\det D\Phi| |(D\Phi)^{-1}\nu_{\Omega^*}|\,d\HH^{d-1}.$$
%We denote by $U$ the set $\{|\det D\Phi| |(D\Phi)^{-1}\nu_\Omega|>1\}$ and we notice that $U\subset\{\Phi\neq Id\}$, which since $\Phi$ is as in \eqref{genlam06cond} implies that $P(\Omega;U)=P(\Omega^\ast;U)$. 
The condition $\|D\Phi-Id\|\leq \eps$ with $\epsilon$ small, implies that for some $C_{d}$ 
$$\left||\det D\Phi||(D\Phi)^{-1}\nu_{\Omega^*}|-1\right|\leq C_{d}\|D\Phi-Id\|_{L^\infty}.$$
Thus, assuming $\Phi=Id$ on $\Om\Delta\Om^*$, we get
\begin{equation*}
P(\Phi(\Omega))-P(\Omega)
=P(\Phi(\Om^*))-P(\Om^*)
% \int_{U\cap\partial^\ast\Omega}\Big(|\det D\Phi||(D\Phi)^{-1}\nu_\Omega|-1\Big)\,d\HH^{d-1}\\
%&\le&\int_{U\cap\partial^\ast\Omega}\Big(|\det D\Phi|^2 |(D\Phi)^{-1}\nu_\Omega|^2-1\Big)\,d\HH^{d-1}\\
%&=&\int_{U\cap\partial^\ast\Omega}\nu_\Omega\cdot \Big(|\det D\Phi|^2 |(D\Phi)^{-2}-1\Big)\nu_\Omega\,d\HH^{d-1}\\
%&\le& P(\Omega;U) C_{d}\|D\Phi-Id\|_{L^\infty}=P(\Omega^\ast;U) C_{d}\|D\Phi-Id\|_{L^\infty}\le 
\leq P(\Omega^\ast) C_{d}\|D\Phi-Id\|_{L^\infty}.
\end{equation*}
\item
For the Dirichlet energy, we directly use the estimate from Lemma \ref{lipen06} where we notice that the constant $C=C_{d,|\Omega|}\|f\|_{L^p}^2$ depends only on the measure of $|\Omega|$, so that one can choose any $\rho>0$ and then  the volume of sets in $\mathcal{A}(\Om^*,\rho)$ is uniformly bounded; thus the constant $C_{d,|\Om|}$ in this class is also bounded.% (as a continuous function of $|\Om|$), which leads to a suitable $C$ so that \eqref{genlam06cond} is valid.}
\item
For the functional $F(\widetilde\lambda_1,\dots,\widetilde\lambda_k)$,  let us first assume for simplicity that $F$ is globally Lipschitz continuous. In this case, by  Lemma \ref{liplam06}, we have that, if $\|D\Phi-Id\|_{L^\infty}\leq 1/2$, 
\begin{align*}
\;\;\;\;\;\;F(\widetilde\lambda_1(\Phi(\Omega)),\dots,\widetilde\lambda_k(\Phi(\Omega)))-F(\widetilde\lambda_1(\Omega),\dots,\widetilde\lambda_k(\Omega))&\le \|\nabla F\|_{L^\infty}C_{d,|\Omega|}\lambda_k(\Omega)\|\Phi-Id\|_{1,\infty}\\
&\le \|\nabla F\|_{L^\infty}C_{d,|\Omega|}\lambda_k(\Omega^\ast\setminus \overline{B_\rho})\|\Phi-Id\|_{1,\infty},
\end{align*} 
where the last inequality is due to the fact that $\Omega\Delta\Omega^\ast\subset B_\rho$ for some ball $B_\rho\subset\R^d$. Now  \cite[Lemma 3]{bulbk} implies that  
\begin{equation}\label{eq:stimelk}
|\lambda_{k}(\Om^*)-\lambda_{k}(\Om^*\setminus\overline{B_{\rho}})|\leq C_{\Om^*}\|w_{\Om^*}-w_{\Om^*\setminus \overline{B_{\rho}}}\|_{H^1}\le C_{\Om^*}\cp(B_{\rho})
\end{equation}
where the first inequality is  \cite[Lemma 3]{bulbk} while the second is given by  \cite[Lemma 3.125]{tesi} (with a possibly different constant $C_{\Om^*}$).
% (see Section \ref{ssect:elliptic} for a definition of $w_{\Om}$) and   \cite[Lemma 3.125]{tesi} allows to control this  term with $C_{\Om^*}\cp(B_{\rho})$ (with a possibly different constant $C_{\Om^*}$). 
 Since  $\ds\lim_{\rho\to0}\cp(B_\rho)=0$, we  see that, choosing $\rho$ small enough, there is a constant $C$ such that \eqref{genlam06cond} holds. The case of a local Lipschitz continuous function $F$ easily follows from \eqref{eq:stimelk}  since it implies that  $|\widetilde\lambda_{k}(\Phi(\Om))-\widetilde\lambda_{k}(\Om^*)|\le 1$  if $\Phi$ is sufficiently close to $Id$ and $\Om\in\mathcal{A}(\Om^*,\rho)$.
\end{itemize}

%{\Rd I don't understand this last sentence: first both references are misleading (compare to the versions I have (from cvgmt)) because there aren't lemma 3 nor 3.125 in the corresponding references. Second, I don't see why we need the capacity of the ball to go to 0. I probably miss something because for me, the last computation implies the validity of \eqref{genlam06cond}.}
\qed

\section{Supersolutions and sets of bounded mean curvature in the viscosity sense}\label{sect:sup}
In this section, we discuss the properties of the sets which are optimal, with respect to exterior perturbations, for functionals of the form \(P(\cdot)+\mu|\cdot|\). Here is the main result of this section that we will need in the proof of Theorem \ref{th:mainD}.

\begin{prop}\label{prop2}
Suppose that $\Dr\subset\R^d$ is a bounded open set with $C^2$ boundary or that $\Dr=\R^d$.
Suppose that the measurable set $\Omega^\ast\subset\Dr$ is a local shape supersolution for the functional $P+\mu|\cdot|$ in $\Dr$, that is to say : there is a constant $r_0>0$ such that
\begin{equation*}\label{eq:sup}
\begin{array}{rl}
\ds P(\Omega^\ast)+\mu|\Omega^\ast|\le P(\Omega)+\mu|\Omega|,&\hbox{for every measurable set }\Omega \text{ with }\\
&\ds\Omega^\ast\subset\Omega\subset\Dr \hbox{ and } \Omega\Delta\Omega^\ast\subset B_{r_{0}}(x_0) \text{ for some } x_0\in\R^d.\\ 
\end{array}
\end{equation*}
Then $\Omega^\ast$ has the following properties:
\begin{enumerate}[(a)]
\item There are constants $r_1>0$ and $\mu_1\in\R$ such that $\Omega^\ast$ is a local shape supersolution in $\R^d$ for the functional $P+\mu_1|\cdot|$ which means that
\begin{equation*}\label{eq:sup1}
\begin{array}{rl}
\ds P(\Omega^\ast)+\mu_1|\Omega^\ast|\le P(\Omega)+\mu_1|\Omega|&\hbox{for every measurable set }\Omega \text{ with}\\
&\ds\Omega^\ast\subset\Omega \hbox{ and } \Omega\Delta\Omega^\ast\subset B_{r_{1}}(x_1) \text{ for some } x_1\in\R^d.\\ 
\end{array}
\end{equation*}
%\item $\Om^\ast$ satisfies an exterior density estimate,\\[-2mm]
\item If we identify the set $\Omega^\ast$ with the set of points of density $1$ (see \eqref{eq:density1}), then $\Omega^\ast$ is open and $H^1_0(\Omega^\ast)=\widetilde H^1_0(\Omega^\ast)$. In particular, $\forall k\in\N^*, \;\lambda_k(\Omega^\ast)=\widetilde\lambda_k(\Omega^\ast)$ and $\forall f\in L^p(D), \;E_f(\Omega^\ast)=\widetilde E_f(\Omega^\ast)$.\\[-2mm]
\item The energy function $w_{\Omega^\ast}$, solution of the equation 
\begin{equation*}\label{eq:-du=1}
-\Delta w_{\Omega^\ast}=1\quad\text{in}\quad\Omega^\ast,\qquad w_{\Omega^\ast}\in H^1_0(\Omega^\ast),
\end{equation*}
is Lipschitz continuous on $\R^d$.
\end{enumerate}
\end{prop}
\begin{proof}
The first claim {\it (a)} follows from Lemma \ref{calibr06} applied to $\Dr$ and  Lemma \ref{transitivepropertysupersol06}. Point {\it (b)} is contained in Lemma \ref{olddens06}.
% and follows from the exterior density estimate for $\Omega^\ast$
The Lipschitz continuity of $w_{\Omega^\ast}$ is proved in Proposition \ref{lipprop06}.
\end{proof}

In what follows, we will revisit the properties of the shape supersolutions and we will also introduce the sets of bounded (from below) curvature in the viscosity sense. 

\subsection{Definitions and first properties of the supersolutions}

\begin{deff}\label{subsoldef}
Let $\mathcal{F}:\mathcal B(\R^d)\to\R\cup\{+\infty\}$ % on the family of Borel sets $\mathcal B(\R^d)$ on $\R^d$
 and let $\Omega\in\mathcal{B}(\R^d)$ be such that $\mathcal{F}(\Omega)<+\infty$. We say that:
\begin{itemize}
\item $\Omega$ is a {\bf supersolution for $\mathcal{F}$}, if
$$\mathcal{F}(\Omega)\le\mathcal{F}(\widetilde\Omega),\quad\hbox{for every measurable set}\quad \widetilde\Omega\supset\Omega.$$
\item $\Omega$ is a {\bf supersolution for $\mathcal{F}$ in the set $\Dr\subset\R^d$}, if 
$$\mathcal{F}(\Omega)\le\mathcal{F}(\widetilde\Omega),\quad\hbox{for every measurable set}\quad \widetilde\Omega\supset\Omega\quad\hbox{such that}\quad \widetilde\Omega\setminus\Omega\subset\Dr.$$
\item $\Omega$ is a {\bf local supersolution for $\mathcal{F}$ in the set $\Dr\subset\R^d$}, if there is a constant $r_0>0$ such that $\Omega$ is a supersolution for $\mathcal{F}$ in $B_{r_0}(x)\cap \Dr$ for every ball $B_{r_0}(x)\subset\R^d$.
\end{itemize} 
\end{deff}

\begin{oss}
Suppose that $\Omega\in\mathcal{B}(\R^d)$ is a (local) supersolution for the functional $\mathcal{F}+\mathcal{G}$ and that $\mathcal{G}:\mathcal{B}(\R^d)\to\R$ is decreasing with respect to the set inclusion. Then $\Omega\in\mathcal{B}(\R^d)$ is a (local)  supersolution also for $\mathcal{F}$. Indeed, it is sufficient to notice that, by the monotonicity of $\mathcal{G}$ and the superoptimality of $\Omega$, we have 
$$\mathcal{F}(\Omega)+\mathcal{G}(\Omega)\le \mathcal{F}(\widetilde{\Omega})+\mathcal{G}(\widetilde{\Omega})\le \mathcal{F}(\widetilde{\Omega})+\mathcal{G}(\Omega),\qquad\forall  \Omega\subset\widetilde{\Omega}.$$
\end{oss}

\begin{oss}
Suppose that $\mathcal{G}:\mathcal{B}(\R^d)\to\R$ is one of the following functionals 
\begin{itemize}
\item $\mathcal{G}(\Omega)=\widetilde E_f(\Omega)$, for some $f\in L^p$ with $p\in[2,\infty]$;
\item $\mathcal{G}(\Omega)=F\big(\widetilde\lambda_1(\Omega),\dots,\widetilde\lambda_k(\Omega)\big)$, where $F:\R^k\to\R$ is a function increasing in each variable. 
\end{itemize}
In both cases, $\mathcal{G}$ is decreasing with respect to the set inclusion and thus every supersolution for the functional $P(\Omega)+\mathcal{G}(\Omega)+\mu|\Omega|$ is also a supersolution for the functional $P(\Omega)+\mu|\Omega|$.
\end{oss}

When we deal with shape optimization problems in a box $\Dr$, \emph{a priori} we can only consider perturbations of a set $\Omega\subset\Dr$, which remain inside the box. The following lemma allows us to eliminate this restriction and work with   the minimizers as if they were solutions of the problem in the free case $\Dr=\R^d$. 

\begin{lemma}\label{transitivepropertysupersol06}
Let $\Omega\subset\Dr$ be two measurable sets in $\R^d$ and let $\mathcal{F}=P+\mu|\cdot|$, where $\mu>0$. If $\Dr$ is a (local) supersolution for $\mathcal{F}$ and $\Omega$ is a (local) supersolution for $\mathcal{F}$ in $\Dr$, then $\Omega$ is a (local) supersolution for $\mathcal{F}$ in $\R^d$.
\end{lemma}
\begin{proof}
We start by recalling  the following formulas for the perimeter of the union and intersection of sets of finite perimeter, see \cite[Section 16.1]{maggi}: for every measurable set $E$ and $F$, 
\[
P(E\cup F)=P(E,F^{(0)})+P(F,E^{(0)})+\mathcal H^{d-1}(\{\nu_E=\nu_F\})
\]
and 
\[
 P(E\cap F)=P(E,F^{(1)})+P(F,E^{(1)})+\mathcal H^{d-1}(\{\nu_E=-\nu_F\}).
 \]
Here \(E^{(0)}=(\R^d\setminus E)^{(1)}\) is the set of density zero point of \(E\) and \(\{\nu_E=\pm\nu_F\}\) is a short hand notation for \(\{\nu_E=\pm\nu_F\}\cap \partial^*E\cap \partial^* F\). Recall also that, for every set of finite perimeter, \(\mathcal H^{d-1}(\R^d\setminus (E^{(0)}\cup E^{(1)}\cup \partial^*E))=0\) and \(E^{(0)}, E^{(1)}\) and 
\(\partial^*E\) are disjoint.
 
Let now $\widetilde\Omega\supset\Omega$. Since $\Dr$ is a supersolution, we get
\[
\begin{split}
P&(\widetilde\Omega;\Dr^{(0)})+P(\Dr;\widetilde\Omega^{(0)})+\mathcal H^{d-1}(\{\nu_\Dr=\nu_{\widetilde\Omega}\})+\mu|\widetilde\Omega\cup\Dr|=\mathcal{F}(\widetilde\Omega\cup\Dr)\ge \mathcal{F}(\Dr)\\
&\hspace{5cm}=P(\Dr;\widetilde\Omega^{(0)})+P(\Dr;\widetilde\Omega^{(1)})+\mathcal H^{d-1}(\partial^*\Dr\cap \partial^*\widetilde \Omega)+\mu|\Dr|,
\end{split}
\]
which gives  (as by definition $\{\nu_\Dr=\nu_{\widetilde\Omega}\}\subset \partial^*\Dr\cap \partial^*\widetilde \Omega$)
\begin{equation}\label{supedDrOmega}
P(\Dr;\widetilde\Omega^{(1)})\le P(\widetilde\Omega;\Dr^{(0)})+\mu|\widetilde\Omega\setminus\Dr|.
\end{equation}
On the other hand, we can test the super-optimality of $\Omega$ with $\widetilde\Omega\cap\Dr$ and then use \eqref{supedDrOmega} to obtain  
\begin{align*}
\mathcal{F}(\Omega)\le \mathcal{F}(\widetilde\Omega\cap\Dr)&=P(\widetilde\Omega;\Dr^{(1)})+P(\Dr;\widetilde\Omega^{(1)})+\mathcal H^{d-1}(\{\nu_\Dr=-\nu_{\widetilde \Omega}\})+\mu|\widetilde\Omega\cap\Dr|\\
&\le P(\widetilde\Omega;\Dr^{(1)})+P(\widetilde\Omega;\Dr^{(0)})+\mathcal H^{d-1}(\partial^*\Dr\cap \partial^*\widetilde \Omega)+\mu|\widetilde\Omega\setminus\Dr|+\mu|\widetilde\Omega\cap\Dr|\\
&= P(\widetilde\Omega)+\mu|\widetilde\Omega|=\mathcal{F}(\widetilde\Omega).
\end{align*}
For the case of local supersolutions, it is enough to consider $\widetilde\Omega$ such that $\widetilde\Omega\setminus\Dr\subset\widetilde\Omega\setminus\Omega\subset B_r(x)$ and then use the same argument as above.
\end{proof}

The following lemma is the first step in the analysis of the supersolutions for $P+\mu|\cdot|$ and shows that they are in fact open sets.

\begin{lemma}\label{olddens06}
Suppose that $\Omega$ is a local supersolution for the functional $\mathcal{F}=P+\mu|\cdot|$. Then:
\begin{enumerate}[(a)]
\item There exists a constant  \(c<1\) such that
$$\frac{|B_r(x)\cap\Omega^c|}{|B_r|}\le c,\qquad \forall x\in\partial^M\Omega,\quad \forall \,r\le r_0,$$
where $r_0$ is as in Definition \ref{subsoldef}.
\item $\Omega$ is an open set.
\item $H^1_{0}(\Om)=\widetilde{H}^1_{0}(\Om)$.
%The Sobolev space $H^1_0(\Omega)$ on the open set $\Omega$ can be characterized as
%$$H^1_0(\Omega)=\Big\{u\in H^1(\R^d):\ u=0\ \ \hbox{a.e. on}\ \ \Omega^c\Big\}.$$
\item The weak solution of the equation 
$$-\Delta w_\Omega=1\quad\text{in}\quad\Omega,\qquad w_\Omega\in H^1_0(\Omega),$$
is H\"older continuous on $\R^d$.
\end{enumerate}
\end{lemma}
\begin{proof}
The proof of {\it (a)} is classical, see for instance \cite{deve}. The proof of {\it (b)} follows by {\it (a)} and the identification $\Omega=\Omega^{(1)}$ defined in \eqref{eq:density1}. The last two claims {\it (c)} and {\it (d)} follow by {\it (a)}: for {\it (c)} see \cite[Proposition 4.7]{deve}, and for {\it (d)}, see \cite[Proposition 4.6]{deve} (see the proof to be convinced that only {\it (a)} is used), which implies \cite[Proposition 5.2]{deve} asserting H\"older regularity for $w_{\Om}$.
\end{proof}

\subsection{Mean curvature bounds in the viscosity sense}
Let us start with the following definition.

\begin{deff}\label{curvisc}
For an open set $\Omega\subset\R^d$ and $c\in\R$, we say that {\bf the mean curvature of $\partial\Omega$ is bounded from below by $c$ in the viscosity sense} ($H_\Omega\ge c$), if for every open set $U\subset\Omega$ with smooth boundary and every point $x_0\in\partial\Omega\cap\partial U$, we have that $H_ U(x_0)\ge c$.
\end{deff}

We will now show that supersolutions have bounded mean curvature  in the viscosity sense. 
%For let us premise the following classical computation, see for instance \cite[Appendix B]{giusti}. If $\Omega\subset\R^d$ is  an open set with smooth boundary, then, up to an orthogonal transformation, we can write it in a neighbourhood of a given boundary point $x_0\in \partial \Omega$ as
%$$\Omega=\{(x_1,\dots,x_d)\in\R^d:\ \phi(x_1,\dots,x_{d-1})>x_d\},$$
%for a smooth function $\phi:\R^{d-1}\to\R$. Thus the mean curvature $H_\Omega$ of $\Omega$ (with respect to the exterior normal)  is (locally) given by
%$$H_\Omega=\hbox{div}\left(\frac{\nabla\phi}{\sqrt{1+|\nabla\phi|^2}}\right)=\sum_{i=1}^{d-1}\frac{\phi_{ii}}{\sqrt{1+|\nabla \phi|^2}}-\sum_{i,j=1}^{d-1}\frac{\phi_i\phi_j\phi_{ij}}{(1+|\nabla\phi|^2)^{3/2}},$$
%and, in particular, if the coordinates $x_1,\dots, x_d$ are such that $|\nabla \phi|(\bar x_0)=0$ (\(x_0=(\bar x_0,(x_0)_d)\)), then we get
%$$H_\Omega(x_0)=\sum_{i=1}^{d-1}\phi_{ii}(\bar x_0)=\Delta_{d-1}\phi(\bar x_0).$$

\begin{prop}
Let $\Omega\subset\R^d$ be an open set of finite measure. If $\Omega$ is a local supersolution for the functional $P+\mu|\cdot|$, then $H_\Omega\ge-\mu$ in the viscosity sense.
\end{prop}
\begin{proof}
Let $ U\subset\Omega$ be an open set with smooth boundary and let $x_0\in\partial U\cap\partial\Omega$. We can suppose that $x_0=0$ and that $ U$ is locally the epigraph of a smooth function $\phi:\R^{d-1}\to\R$ such that $\phi(0)=|\nabla\phi(0)|=0$. We can now suppose that $\{0\}=\partial U\cap\partial\Omega$, up to replacing $ U$  by a smooth set $\widetilde U\subset U$, which is locally the epigraph of the function $\widetilde\phi(x)=\phi(x)+|x|^4$. We now consider the family of sets $ U_\eps=-\eps e_d+\widetilde U$, where $e_d=(0,\dots,0,1)$. By the choice of $\widetilde U$, for every $r>0$, one can find $\eps_0>0$ such that 
$$ U_\eps\setminus\Omega\subset U_\eps\setminus U\subset B_r,\  \hbox{for every}\ \ 0<\eps<\eps_0.$$ 
Thus one can use the sets $\Omega_\eps= U_\eps\cup\Omega$ to test the local superminimality of $\Omega$. Let $d_\eps: U_\eps\to\R$ be the distance function
$$d_\eps(x)=\hbox{dist}(x,\partial U_\eps).$$ 
For small enough $\eps$, we have that $d_\eps$ is smooth in $ U_\eps\cap B_r$, up to the boundary $\partial U_\eps$. By  \cite[Appendix B]{giusti}, we have that $H_ U(0)=H_{ U_\eps}(-\eps e_d)=\Delta d_\eps(-\eps e_d)$. If $H_ U(0)< -\mu$, then for $\eps$ small enough, we can suppose that $\Delta d_\eps<-\mu$ in $ U_\eps\cap B_r$. Thus, denoting by $\nu_\Omega$ the exterior normal to a set of finite perimeter $\Omega$, we have
\begin{align*}
-\mu| U_\eps\setminus\Omega|&>\int_{ U_\eps\setminus\Omega}\Delta d_\eps(x)\,dx\\
&=\int_{\Omega\cap\partial U_\eps}\nabla d_\eps\cdot\nu_{ U_\eps}\,d\HH^{d-1}-\int_{ U_\eps\cap\partial\Omega}\nabla d_\eps\cdot\nu_\Omega\,d\HH^{d-1}\ge  P( U_\eps;\Omega)-P(\Omega; U_\eps),
\end{align*}
which implies 
$$P(\Omega)+\mu|\Omega|>P(\Omega\cup U_\eps)+\mu|\Omega\cup U_\eps|,$$
thus contradicting the local superminimality of $\Omega$.
\end{proof}

\begin{oss}
The converse is in general false. Indeed, the set $\Omega$ on Figure \ref{fignonvisc} has mean curvature bounded from below in the viscosity sense. On the other hand it is not a supersolution for $P+\mu|\cdot|$ since, adding a ball $B_r(x_0)$ in the boundary point $x_0\in\partial\Omega$, decreases the perimeter linearly $P(\Omega)-P(\Omega\cap B_r)\sim r$.

\begin{figure}
\includegraphics[scale=0.4]{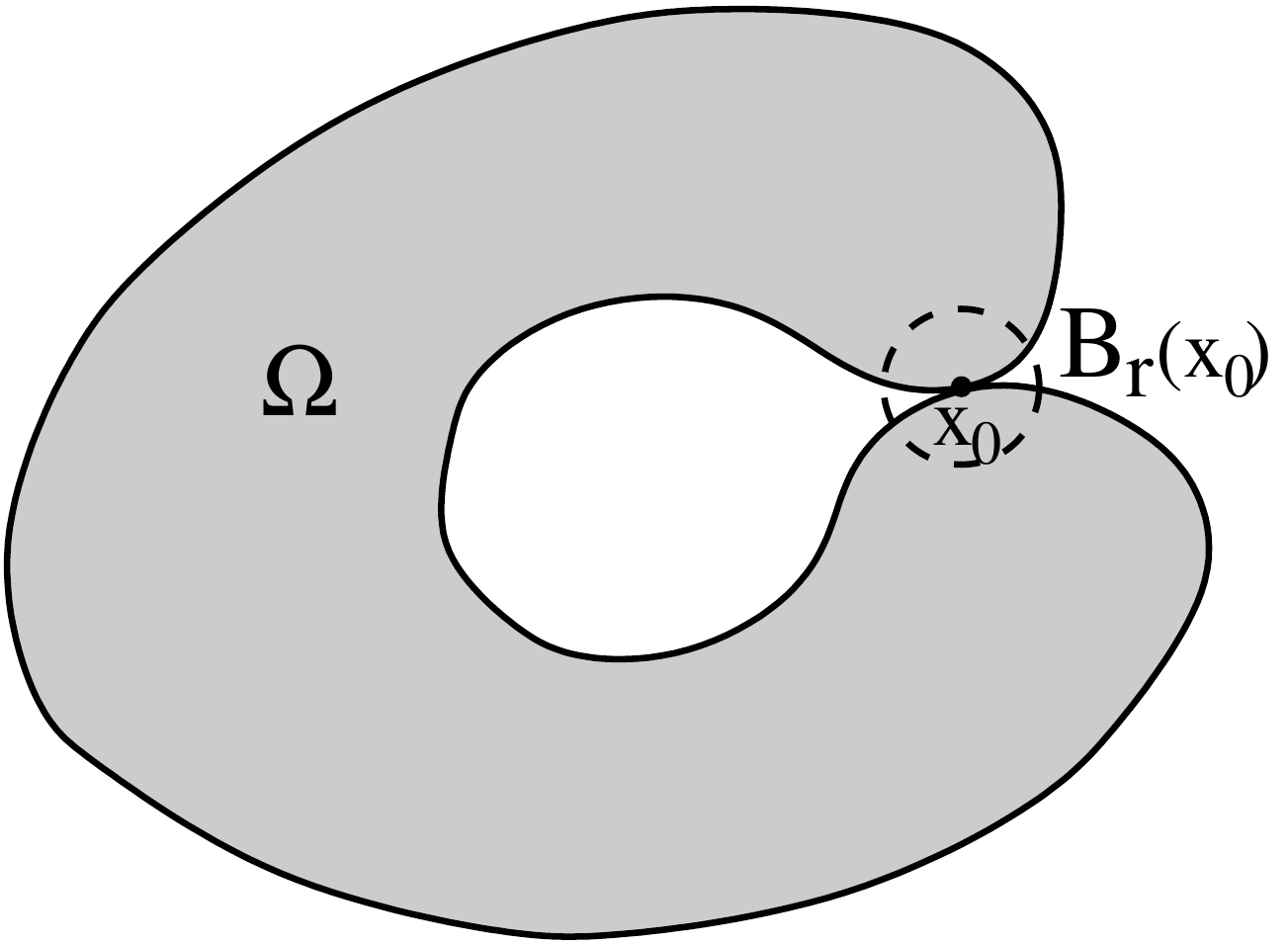}
\caption{$\Omega$ has mean curvature bounded from below in the viscosity sense, but is not a local supersolution for $P+\mu|\cdot|$.}
\label{fignonvisc}
\end{figure}
\end{oss}

The following lemma is a generalization of \cite[Lemma 5.3]{deve}.

\begin{lemma}\label{viscdist}
Suppose that $\Omega$  is an open set such that $H_\Omega\ge-\mu$ in the viscosity sense. Then the distance function $d_\Omega(x)=\hbox{dist}(x,\partial\Omega)$ satisfies $\Delta d_\Omega\le \mu$ in the viscosity sense.
\end{lemma}
\begin{proof}
Suppose that $\vf\in C^\infty_c(\Omega)$ is such that $\vf\le d_\Omega$ and suppose that $x_0\in\Omega$ is such that $\vf(x_0)=d_\Omega(x_0)$. In what follows, we set $t=\vf(x_0)$, $\Omega_t=\{\vf>t\}\subset\{d_\Omega>t\}$ and $n=\frac{x_0-y_0}{|x_0-y_0|}$, where  $y_0\in\partial\Dr$ is chosen such that $|x_0-y_0|=t$ (see Figure \ref{figviscdist}).
\begin{figure}
\includegraphics[scale=0.5]{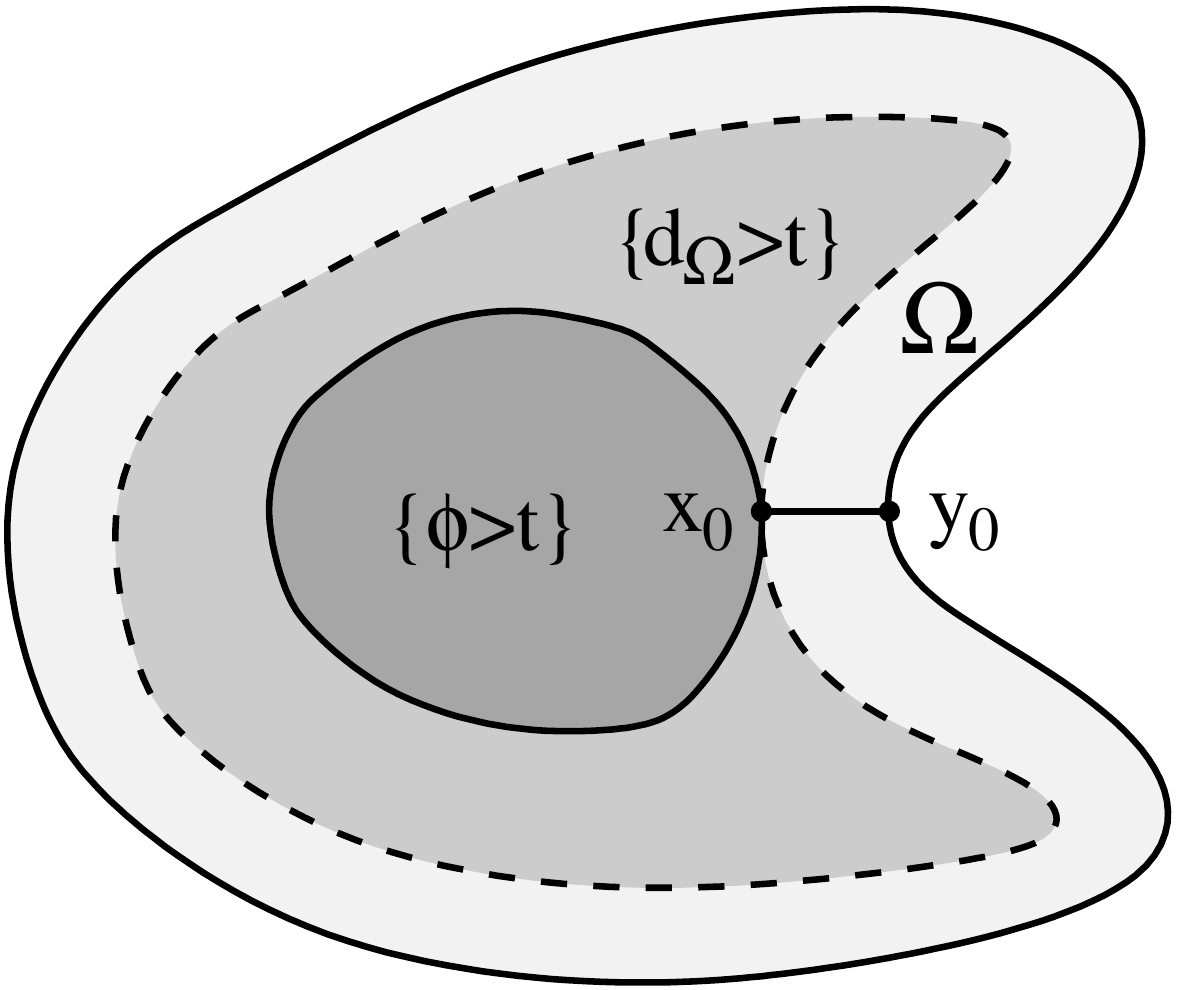}
\caption{Testing the viscosity bound $H_\Omega\ge-\mu$ with the set $\{\vf>t\}$.}
\label{figviscdist}
\end{figure}
We first prove that $\nabla\vf(x_0)=n$. Indeed, on one hand the Lipschitz continuity of $d_\Omega$ gives
$$\vf(x)-\vf(x_0)\le d_\Omega(x)-d_\Omega(x_0)\le |x-x_0|,$$
and so $|\nabla\vf|(x_0)\le 1$. On the other hand, we have 
$$\vf(x_0)-\vf(x_0+\eps n)\ge d_\Omega(x_0)-d_\Omega(x_0+\eps n)=\eps,$$
which gives $|\nabla \vf|(x_0)\ge \frac{\partial\vf}{\partial n}(x_0)=1$.

We now notice that $\vf$ is concave in the direction of $n$. Indeed 
\begin{align*}
\frac{\partial^2\vf}{\partial n^2}(x_0)&=\lim_{\eps\to 0^+}\frac{\vf(x_0+\eps n)+\vf(x_0-\eps n)-2\vf(x_0)}{\eps^2}\\
&\le\lim_{\eps\to 0^+}\frac{d_\Omega(x_0+\eps n)+d_\Omega(x_0-\eps n)-2d_\Omega(x_0)}{\eps^2}\\
&\le\lim_{\eps\to 0^+}\frac{(t+\eps)+(t-\eps)-2t}{\eps^2}=0.
\end{align*}
Since $|\nabla\vf|(x_0)=1$, the level set $\Omega_t$ has smooth boundary in a neighbourhood of $x_0$ and $n=-\nu_{\Omega_t}(x_0)$ is the interior normal at $x_0\in\partial\Omega_t$. Then we have 
$$\Delta\vf(x_0)=\frac{\partial^2\vf}{\partial n^2}(x_0)-\frac{\partial \vf}{\partial n}(x_0)H_{\Omega_t}(x_0)\le -H_{\Omega_t}(x_0).$$
On the other hand, setting $ U=tn+\Omega_t$, we have $ U\subset tn+\{d_\Omega>t\}\subset\Omega$, $y_0\in\partial U$ and $H_ U(y_0)=H_{\Omega_t}(x_0)\ge-\mu$, which gives $\Delta \vf (x_0)\le \mu$ and concludes the proof.
\end{proof}
In the following proposition, we prove the main result of this section. We state it for local shape supersolutions $\Omega$, but the main ingredients of the proof are continuity of the energy function $w_\Omega$ and the fact that $H_\Omega$ is bounded from below in the viscosity sense. 

\begin{prop}\label{lipprop06}
Suppose that $\Omega\subset\R^d$ is a local supersolution for the functional $\mathcal{F}=P+\mu|\cdot|$. Then $\Omega$ is an open set and the energy function $w_\Omega$ is Lipschitz continuous on $\R^d$ with a constant depending only on $\mu$, the dimension $d$ and the measure $|\Omega|$.
\end{prop}
\begin{proof}
Recall that, by Lemma \ref{olddens06},  \(\Omega\) is open and \(w_\Omega\) is continuous. Let us set  $w=w_\Omega$ for simplicity. Consider the function
$$h(t)=\frac{1-e^{-Mt}}{N},\qquad\hbox{where}\qquad M=1+|\mu|\qquad\hbox{and}\qquad N=\min\Big\{\frac{M}2 e^{-M\left({|\Omega|}/{ U_d}\right)^{1/d}},\frac{d U_d^{2/d}}{|\Omega|^{2/d}}\Big\}.$$
By construction, $h$ is $M/N$-Lipschitz  continuous and is a homeomorphism $h:[0,+\infty)\to [0,1/N)$. We will show that the following inequality holds:
\begin{equation}\label{mainestlip}
w(x)\le h(d_\Omega(x)),\qquad\forall x\in\Omega.
\end{equation}
We first note that, since $\|w\|_{L^\infty}\le \frac{|\Omega|^{2/d}}{2d \omega_d^{2/d}}<1/N$ (see \eqref{winftyf=1}) the function $h^{-1}(w)$ is well defined, positive and has the same regularity as $w$. Suppose there exists $\eps>0$ such that the function $w_\eps:=(w-\eps)^+$ satisfies
\begin{equation}\label{epsabsurd}
w_\eps\le h(d_\Omega)\quad \hbox{in}\quad\Omega\qquad\hbox{and}\qquad w_\eps(x_0)=h(d_\Omega(x_0)),\quad \hbox{for some}\quad x_0\in\Omega.
\end{equation}
Then considering the function $u_\eps=h^{-1}(w_\eps)$, we get
$$u_\eps\le d_\Omega\quad \hbox{in}\quad\Omega\qquad\hbox{and}\qquad u_\eps(x_0)=d_\Omega(x_0).$$
By Lemma \ref{viscdist}, we have $\Delta u_\eps(x_0)\le \mu$ and $|\nabla u_\eps|^2(x_0)=1$ so that
\begin{align}
-1=\Delta w(x_0)&=h''(u_\eps(x_0))|\nabla u_\eps|^2(x_0)+h'(u_\eps(x_0))\Delta u_\eps(x_0)\nonumber\\
&\le-\frac{M^2}{N}e^{-Mu_\eps(x_0)}+\frac{ M}{N}e^{-Mu_\eps(x_0)}\mu= \frac{M}{N}(\mu-M)e^{-Mu_\eps(x_0)}\nonumber\\
&\le -\frac{M}{N}e^{-Mu_\eps(x_0)}\le -2,\label{mainbndeps}
\end{align}
where the last inequality is due to the fact that $u_\eps\le d_\Omega\le \left({|\Omega|}/{\omega_d}\right)^{1/d}$ and to the definition of $N$.
Now since \eqref{mainbndeps} is a contradiction, it implies that \eqref{epsabsurd} cannot be true either and we therefore obtain \eqref{mainestlip}. In particular, this gives that 
$$w(x)\le \frac{M}{N}d_\Omega(x),\qquad\forall x\in\Omega,$$
which roughly speaking corresponds to a gradient estimate $|\nabla w|\le M/N$ on the boundary $\partial\Omega$. There are several very well known ways to extend this estimate inside $\Omega$. We recall the method of Brezis-Sibony \cite{bresib}, which is an elegant way to avoid the regularity issues of $w$ and $\Omega$. Indeed, we recall that $w$ is the unique minimizer in $H^1_0(\Omega)$ of the functional $\ds J(u)=\frac12\int_\Omega|\nabla u|^2\,dx-\int_\Omega u\,dx$ and we test the minimality of $w$ against the functions 
$$w_+(x)=\big(w(x+s)+\frac{M}{N}|s|\big)\wedge w(x)\qquad\hbox{and}\qquad w_-(x)= \big(w(x-s)-\frac{M}{N}|s|\big)\vee w(x),$$
where $s\in \R^d$ is arbitrary. In fact, we can use $w_{\pm}$ as test functions since $w_{\pm}\le \frac{M}{N}d_\Omega$, which gives that $w_{\pm}\in H^1_0(\Omega)$. Now the inequalities $J(w_+)\ge J(w)$ and $J(w_-)\ge J(w)$ give respectively 
\begin{equation}\label{Epm04}
\begin{array}{ll}
\ds\frac12\int_{E_+}|\nabla w(x+s)|^2\,dx-\int_{E_+}\big(w(x+s)+\frac{M}{N}|s|\big)\,dx\le \frac12\int_{E_+}|\nabla w(x)|^2\,dx-\int_{E_+}w(x)\,dx,\\
\\
\ds\frac12\int_{E_-}|\nabla w(x-s)|^2\,dx-\int_{E_-}\big(w(x-s)-\frac{M}{N}|h|\big)\,dx\le \frac12\int_{E_-}|\nabla w(x)|^2\,dx-\int_{E_-}w(x)\,dx,
\end{array}
\end{equation}
where $E_+=\{w_+<w\}$ and $E_-=\{w_->w\}$. Now we notice that $E_+=s+E_-$ and, after a change of variables, we obtain that both inequalities in \eqref{Epm04} are in fact equalities which give $J(w)=J(w_+)=J(w_-)$ and, by the strict convexity of $J$, $w=w_+=w_-$. Since this is true for every $h\in\R^d$, we get that $w_\Omega$ is $M/N$-Lipschitz on $\R^d$,  and in particular this implies that $\Om=\{w_{\Om}>0\}$ is open.
\end{proof}

The main point of this section is that the property of being a supersolution of $P+\mu|\cdot|$ corresponds to a curvature bound, which we may then use to obtain the regularity of the solutions of some elliptic PDEs. We conclude this section with the converse implication, i.e. that the regular sets whose curvature is bounded from below are in fact local supersolutions for a functional of the form $P+\mu|\cdot|$.

\begin{lemma}\label{calibr06}
Suppose that $\Omega\subset \R^d$ is an open set with $C^2$ boundary such that $H_\Omega\ge -\mu$. Then $\Omega$ is a local supersolution for the functional $\mathcal F=P+\mu |\cdot|$.
\end{lemma}
\begin{proof}
We will prove the proposition by constructing an appropriate calibration $\xi$. Let $x_{0}\in\partial\Om$ and assume that in a  neighbourhood $V_{x_0}$ of $x_0$, $\partial\Omega$ is the epigraph of a function $\phi: \R^{d-1}\to \R$. Consider the function $\xi:V_{x_0}\to \R^d$ defined by 
$$\xi(x_1, \dots, x_d)=\frac{(\nabla_{d-1}\phi,-1)}{\sqrt{1+|\nabla_{d-1}\phi|^2}},\qquad\hbox{where}\qquad \nabla_{d-1}{\phi}=\Big(\frac{\partial \phi}{\partial x_1},\dots, \frac{\partial \phi}{\partial x_{d-1}}\Big).$$
It is straightforward to check that 
\begin{itemize}
\item $|\xi|\le 1$ in $V_{x_0}$ and the restriction of $\xi$ to $\partial\Omega$ is precisely the normal vector field to $\partial\Omega$.
\item a straightforward calculation gives that 
\begin{equation*}\label{divcalib04}
\ds\text{div } \xi (x_1,\dots,x_d)=\frac{\Delta_{d-1}\phi}{\sqrt{1+|\nabla_{d-1}\phi|^2}}-\frac{\sum_{i,j=1}^{d-1}\phi_i\phi_j\phi_{ij}}{\big(1+|\nabla_{d-1}\phi|^2\big)^{3/2}}=-H_\Omega(x_1,\dots,x_{d-1}).
\end{equation*}
\end{itemize}
Since all sets of finite perimeter can be approximated by smooth sets, it is sufficient to show that the property of being a supersolution holds for $C^2$ sets. 
For an arbitrary $C^2$ set $\widetilde\Omega\supset\Omega$ such that $\Omega\Delta\widetilde\Omega\subset V_{x_0}$, we get 
\begin{align*}
P(\Omega;V_{x_0})=\int_{\partial\Omega} \xi\cdot \nu_\Omega\,d\HH^{d-1}&=\int_{\partial\widetilde\Omega} \xi\cdot \nu_{\widetilde\Omega}\,d\HH^{d-1}+\int_{\widetilde\Omega\setminus\Omega}\text{div}\xi\,dx\\
&\le P(\widetilde\Omega; V_{x_0})+\int_{\widetilde\Omega\setminus\Omega}\text{div}\xi\,dx\le P(\widetilde\Omega; V_{x_0})+\mu |\widetilde\Omega\setminus\Omega|,
\end{align*}
where $\nu_\Omega$ and $\nu_{\widetilde\Omega}$ are the exterior normals to $\Omega$ and $\widetilde\Omega$. Thus $\Omega$ is a local supersolution for $\mathcal{F}$.
\end{proof}

\section{Shape subsolutions for functionals involving the perimeter}\label{sect:sub}

In this section, we prove the following proposition.

\begin{prop}\label{prop3}
 Let us consider $\mathcal{G}:\mathcal{B}(\R^d)\to\R$ to be one of the following functionals:
\begin{itemize}
\item $\widetilde{\mathcal{G}}(\Omega)=\widetilde E_f(\Omega)$, for $f\in L^p(\Dr)$ with $p\in(d,\infty]$.
\item $\widetilde{\mathcal{G}}(\Omega)=F\big(\widetilde\lambda_1(\Omega),\dots,\widetilde\lambda_k(\Omega)\big)$, where the function $F:\R^k\to\R$ is locally H\"older continuous with exponent $\beta>0$.
\end{itemize}
Suppose that $\Omega^\ast\subset\R^d$ is an open set of finite Lebesgue measure such that the energy function $w_{\Omega^\ast}$, solution of \eqref{eq:-du=1}, is Lipschitz continuous on $\R^d$, and satisfies:
\begin{equation}\label{subomegaast06}
P(\Omega^*)+\widetilde \G(\Om^*)+\mu|\Om^*|\le P(\Omega)+\widetilde \G(\Om)+\mu|\Om|,\ \hbox{ for every measurable set}\quad \Omega\subset\Omega^*,
\end{equation}
for some fixed $\mu\in\R$.\\ 
Then $\Omega^\ast$ is a local interior quasi-minimizer of the perimeter with an exponent $d\beta$ where $\beta=1-1/p$ if $\widetilde{\mathcal{G}}=\widetilde E_f$ and $\beta$ is the H\"older exponent of $F$ if $\widetilde{\mathcal{G}}(\Om)=F\big(\widetilde\lambda_1(\Omega),\dots,\widetilde\lambda_k(\Omega)\big)$. Thus there are constants $r_0>0$ and $C>0$ such that
\begin{equation*}\label{eq:quasi2}
\begin{array}{ll}
P(\Om^*)\leq P(\Om)+Cr^{d\beta},&\textrm{for every measurable }\Om\subset\Om^*\textrm{ with }\\
&\Om^*\Delta\Om\subset B_{r}(x_{0})\textrm{ for some } r<r_{0}\textrm{ and }x_0\in\R^d.
\end{array}
\end{equation*}
\end{prop}
\begin{proof}
It is sufficient to apply Lemma \ref{gammaholdenergy06} and Lemma \ref{subconclusivelemma06}.
\end{proof}

\begin{oss}\label{remark3}
In the case where $\widetilde{\mathcal{G}}=F(\widetilde\lambda_1,\dots,\widetilde\lambda_k)$, for $F:\R^k\to\R$ which is locally Lipschitz continuous, this result was proved in \cite{deve} in the particular case $\mu=0$. 
%In this case the quasi-minimality exponent is precisely $d$, i.e. there are constants $r_0>0$ and $C>0$ such that
%\begin{equation}\label{eq:quasi206}
%\begin{array}{ll}
%P(\Om^*)\leq P(\Om)+Cr^{d},&\textrm{for every measurable }\Om\subset\Om^*\textrm{ such that }\\
%&\Om^*\Delta\Om\subset B_{r}(x_{0})\textrm{ for some } r<r_{0}\textrm{ and }x_0\in\R^d.
%\end{array}
%\end{equation}
\end{oss}

The sets that satisfy an inequality of the form \eqref{subomegaast06} are called shape subsolutions. More generally, let $\mathcal{A}$ be a family of sets and $\mathcal{F}:\mathcal{A}\to\R$ a given functional on $\mathcal{A}$. 

\begin{deff}\label{supsoldef}
We say that the set $\Omega\subset\mathcal{A}$ is a subsolution for $\mathcal{F}$, if the following sub-optimality condition holds:
$$\mathcal{F}(\Omega)\le \mathcal{F}( U),\quad\hbox{for every set}\quad U\in\mathcal{A}\quad\text{such that}\quad  U\subset\Omega.$$
\end{deff}

In what follows, we will suppose that $\mathcal{A}$ is the family of measurable subsets of $\R^d$. We will deduce some qualitative properties of a set $\Omega\subset\R^d$, assuming that $\Omega$ is only a subsolution for a functional of the form $\mathcal{F}(\Omega)=P(\Omega)+\mathcal{G}(\Omega)$, where $\mathcal{G}$ is an energy or spectral functional. In order to obtain a general theory, easy to handle, we introduce the notion of a \emph{$\gamma$-H\"older functional} \footnote{Here $\gamma$-H\"older refers to the topology of \(\gamma\)-convergence, see \cite{bubu05}, and not to the value of the H\"older exponent} in order to transfer the sub-optimality information of the functional $\mathcal G$ to the Dirichlet energy $\widetilde E_1$. We then study the qualitative properties of the sets $\Omega$, satisfying a suitable sub-optimality condition involving the perimeter $P$ and the Dirichlet energy $\widetilde E_1$.  

\subsection{Decreasing $\gamma$-H\"older functionals}\label{ssect:decgamma}

\begin{deff}\label{gamholdef}
We say that the decreasing functional $\mathcal{G}:\mathcal{B}(\R^d)\to\R$ is
\begin{itemize}
\item {\bf $\gamma$-H\"older}, if there is a constant $\beta>0$ such that, for every $\Omega\subset\R^d$ of finite Lebesgue measure, there exists a constant $C>0$ with the following property:
\begin{equation}\label{gamholdefeq}
\mathcal{G}( U)-\mathcal{G}(\Omega)\le C\Big(\widetilde E_1( U)-\widetilde E_1(\Omega)\Big)^\beta,\qquad\forall  U\subset\Omega.
\end{equation}
\item {\bf locally $\gamma$-H\"older}, if there is a constant $\eps>0$ such that \eqref{gamholdefeq} holds for the measurable sets $ U\subset\Omega$ with $\ds \widetilde E_1( U)-\widetilde E_1(\Omega)\le \eps$.
\end{itemize}
\end{deff}

\begin{lemma}\label{gammaholdenergy06}
Suppose that $p>d/2$ and that $f\in L^p(\R^d)$ is a given function. Then, the functional $\widetilde E_f:\mathcal{B}(\R^d)\to\R$ is $\gamma$-H\"older. More precisely, for any measurable set $\Omega\subset\R^d$ of finite Lebesgue measure, we have 
\begin{equation*}
\widetilde E_f( U)-\widetilde E_f(\Omega)\le C\|f\|_{L^p}^2	\Big(\widetilde E_1( U)-\widetilde E_1(\Omega)\Big)^\beta,\qquad\forall  U\subset\Omega,
\end{equation*}
where $\beta=(1-1/p)$ and $C$ is a constant depending on the exponent $p$, the dimension $d$ and the measure $|\Omega|$.
\end{lemma}
\begin{proof}
We first note that we can suppose that $f$ is nonnegative, since the inequality 
\begin{equation*}
\widetilde E_f( U)-\widetilde E_f(\Omega)\le \widetilde E_{|f|}( U)-\widetilde E_{|f|}(\Omega),
\end{equation*}
holds. Indeed, using the definition of $\widetilde E_f$ and the positivity of the operator $R_\Omega-R_ U$ (which follows by the inclusion \(U\subset \Omega\)), we have
\begin{align*}
\widetilde E_f( U)-\widetilde E_f(\Omega)&=\frac12\int_{\Omega}f\big(R_\Omega(f)-R_ U(f)\big)\,dx\\
&=\frac12\int_{\Omega}f_+\big(R_\Omega(f)-R_ U(f)\big)\,dx+\frac12\int_{\Omega}f_-\big(R_\Omega(-f)-R_ U(-f)\big)\,dx\\
&\le\frac12\int_{\Omega}f_+\big(R_\Omega(f_+)-R_ U(f_+)\big)\,dx+\frac12\int_{\Omega}f_-\big(R_\Omega(f_-)-R_ U(f_-)\big)\,dx\\
&\le \widetilde E_{|f|}( U)-\widetilde E_{|f|}(\Omega).
\end{align*}
Since we can suppose $f\ge 0$, we have $R_\Omega(f)-R_ U(f)\ge 0$. We now use an estimate from the proof of \cite[Lemma 3.6]{buc00}, which we sketch for the sake of completeness. For every nonnegative $\phi\in L^p(\Omega)$, we have
\begin{align*}
\int_\Omega |R_\Omega(\phi)-R_ U(\phi)|^p\,dx&\le \|R_\Omega(\phi)-R_ U(\phi)\|_{L^\infty}^{p-1}\int_\Omega \big(R_\Omega(\phi)-R_ U(\phi)\big)\,dx\\
&= \|R_\Omega(\phi)-R_ U(\phi)\|_{L^\infty}^{p-1}\int_\Omega \phi\big(R_\Omega(1)-R_ U(1)\big)\,dx\\
&\le \|R_\Omega(\phi)-R_ U(\phi)\|_{L^\infty}^{p-1}\|\phi\|_{L^p}\|R_\Omega(1)-R_ U(1)\|_{L^{p'}}.
\end{align*}
Now, using the estimate \eqref{winfty06} that we recall here:
$$\|R_\Omega(\phi)\|_{L^\infty}\le \frac{C_d}{2/d-1/p}|\Omega|^{2/d-1/p}\|\phi\|_{L^p},$$
we get from a duality argument that there is a constant $C$ depending on $d, p$ and $|\Omega|$ such that 
$$\|R_{\Omega}-R_ U\|_{\mathcal{L}(L^p(\Omega);L^p(\Omega))}\le C \|R_\Omega(1)-R_ U(1)\|_{L^{p'}}^{1/p}.$$
Now, since $R_\Omega-R_ U$ is a continuous self-adjoint operator on $L^2(\Omega)$, we get that 
$$\|R_{\Omega}-R_ U\|_{\mathcal{L}(L^{p'}(\Omega);L^{p'}(\Omega))}\le C \|R_\Omega(1)-R_ U(1)\|_{L^{p'}}^{1/p}.$$
We can now estimate the difference of the energies as follows:
\begin{align*}
\widetilde E_f( U)-\widetilde E_f(\Omega)=\frac12\int_{\Omega}f\big(R_\Omega(f)-R_ U(f)\big)\,dx
&\le \frac12\|f\|_{L^p}\|f\|_{L^{p'}}\|R_{\Omega}-R_ U\|_{\mathcal{L}(L^{p'}(\Omega);L^{p'}(\Omega))}\\
&\le \frac{C}2\|f\|_{L^p}^2\|R_\Omega(1)-R_ U(1)\|_{L^{p'}}^{1/p},
\end{align*}
which concludes the proof.
\end{proof}

\begin{lemma}
Suppose that $F:\R^k\to\R$ is a locally H\"older continuous function with exponent $\beta>0$. Then the functional $\widetilde{\mathcal{G}}:\mathcal{B}(\R^d)\to\R$ defined as
$$\widetilde{\mathcal{G}}(\Omega):=F\big(\widetilde\lambda_1(\Omega),\dots,\widetilde\lambda_k(\Omega)\big),$$ 
is locally $\gamma$-H\"older with the same H\"older exponent as $F$, i.e.
$$\widetilde{\mathcal{G}}( U)-\widetilde{\mathcal{G}}(\Omega)\le C\big(\widetilde E_1( U)-\widetilde E_1(\Omega)\big)^\beta,\qquad \forall  U\subset\Omega\quad\hbox{s.t.}\quad \widetilde E_1( U)-\widetilde E_1(\Omega)\le \eps,$$
where $C$ and $\eps$ are constants depending on $d$, $k$, $\widetilde{\lambda}_k(\Omega)$, $|\Omega|$ and $\beta$.
\end{lemma}
\begin{proof}
Let $\Omega\subset\R^d$ be a given measurable set of finite measure and let $ U\subset\Omega$. By \cite[Lemma 3]{bulbk}, we have the estimate
\begin{equation*}
 \forall i\in\llbracket 1,k\rrbracket,\;\;\;\widetilde \lambda_i(\Omega)^{-1}-\widetilde \lambda_i( U)^{-1}\le C_B\big(\widetilde E_1( U)-\widetilde E_1(\Omega)\big),
\end{equation*}
where $C_B$ is a constant depending on the dimension $d$, $k$, $\lambda_k(\Omega)$ and the measure $|\Omega|$.

By the local H\"older continuity of $F$, we have constants $C_F>0$ and $\beta>0$ such that 
$$F\big(\widetilde\lambda_1( U),\dots,\widetilde \lambda_k( U)\big)-F\big(\widetilde \lambda_1(\Omega),\dots,\widetilde \lambda_k(\Omega)\big)\le C\sum_{i=1}^k\big(\widetilde\lambda_i( U)-\widetilde\lambda_i(\Omega)\big)^\beta$$
$$= C_F\sum_{i=1}^k\widetilde\lambda_i( U)^\beta\widetilde\lambda_i(\Omega)^\beta\big(\widetilde\lambda_i(\Omega)^{-1}-\widetilde\lambda_i(\Omega)^{-1}\big)^\beta$$
$$\le C_F C_B^\beta\left( \sum_{i=1}^k\widetilde\lambda_i( U)^\beta\widetilde\lambda_i(\Omega)^\beta\right)
\big(\widetilde E_1( U)-\widetilde E_1(\Omega)\big)^\beta$$
$$\le C_F C_B^\beta k^\beta 2^{\beta k}\widetilde\lambda_k(\Omega)^{2\beta}\big(\widetilde E_1( U)-\widetilde E_1(\Omega)\big)^\beta,$$
where the last inequality holds for $ U\subset\Omega$ such that $\ds C_B\big(\widetilde E_1( U)-\widetilde E_1(\Omega)\big)\le \widetilde\lambda_1(\Omega)/2.$
\end{proof}

\subsection{Boundedness of the subsolutions}\label{ssect:boundedness}

In this section, we prove:

\begin{prop}\label{prop:boundedness}
Suppose that the measurable set of finite measure $\Omega^*\subset\R^d$ is a subsolution for the functional $P+\widetilde{\mathcal{G}}+\mu|\cdot|$, where $\mu\in\R$, $\widetilde{\mathcal{G}}=E_{f}$ with $f\in L^p(\R^d)$ and $p>d$, or $\widetilde{\mathcal{G}}=F\big(\widetilde\lambda_1,\dots,\widetilde\lambda_k\big)$ with $F:\R^k\to\R$ being locally H\"older continuous with exponent $\beta>1-\frac{1}{d}$. Then $\Omega^*$ is bounded.\\
In particular, solutions to \eqref{eq:P+Ef4} when $D=\R^d$ are bounded.
\end{prop}
\begin{proof}
The first part of this result is a consequence of the next two lemmas and Section \ref{ssect:decgamma}, and the last part follows by applying Proposition \ref{prop1} asserting that solutions to \eqref{eq:P+Ef4} are subsolutions to $P+\widetilde{\mathcal{G}}+\mu|\cdot|$.
\end{proof}

The following lemma is implicitly contained in \cite[Lemma 3.7]{deve} and was proved in \cite{bubuve2} for general capacitary measures. We state here the result in the case of measurable sets, though we do not reproduce the proof which is exactly similar.

\begin{lemma}\label{planecut}
Suppose that $\Omega$ is a set of finite measure and that $H$ is a half-space in $\R^d$. Then we have
$$\widetilde E_1(\Omega\cap H)-\widetilde E_1(\Omega)\le \sqrt{2\|w_\Omega\|_\infty}\int_{\partial H}w_\Omega\,d\HH^{d-1}-\frac12\int_{H^c}|\nabla w_\Omega|^2\,dx+\int_{H^c}w_\Omega\,dx,$$
where $w_\Omega$ is the energy function on $\Omega$. 
\end{lemma}

%{\Rd An important consequence of the previous statement is the following result, allowing to prove boundedness of subsolutions:}
\begin{lemma}
Consider a constant $\beta\in \big(1-\frac{1}{d},1\big]$, where $d$ is the dimension of the space. Suppose that $\Omega\subset\R^d$ is a measurable set such that
\begin{equation}\label{holdersub}
P(\Omega)-P( U)\le \Lambda\Big(\widetilde E_1( U)-\widetilde E_1(\Omega)\Big)^\beta+\mu|\Omega\setminus U|,\qquad \forall  U\subset\Omega\quad\hbox{such that}\quad \widetilde E_1( U)-\widetilde E_1(\Omega)\le \eps,
\end{equation}
where $\Lambda>0$, $\eps>0$ and $\mu\ge 0$ are given constants. Then $\Omega$ is bounded.
\end{lemma}
\begin{proof}
For every $t\in\R$, we set $H_t=\{(x_1,\dots,x_d)\in\R^d:\ x_1<t\}$. We notice that Lemma \ref{planecut} implies that, for $t$ large enough, $\widetilde E_1(\Omega\cap H_t)-\widetilde E_1(\Omega)\le \eps$. We now use $\Omega\cap H_t$ to test \eqref{holdersub}. By Lemma \ref{planecut} and the bound $\|w_\Omega\|_\infty\le C_d|\Omega|^{2/d}$, we have 
\begin{equation}\label{holdersuneq1}
\begin{array}{ll}
\ds P(\Omega;H_t^c)-P(H_t;\Omega)&\ds \le \Lambda C\big(P(H_t;\Omega)+|\Omega\setminus H_t|\big)^\beta+\mu|\Omega\setminus H_t|\\
\\
&\ds\le\Lambda C\big(P(H_t;\Omega)^\beta+|\Omega\setminus H_t|^\beta\big)+\mu|\Omega\setminus H_t|,
\end{array}
\end{equation}
where $C$ is a constant depending only on the dimension $d$, $\beta$ and $|\Omega|$. On the other hand, by the isoperimetric inequality for $\Omega\setminus H_t$, we get
\begin{equation}\label{holdersuneq2}
|\Omega\setminus H_t|^{\frac{d-1}{d}}\le C_d \big(P(\Omega;H_t^c)+P(H_t;\Omega)\big),
\end{equation}
for a dimensional constant $C_d>0$. Substituting the estimate for $P(\Omega;H_t^c)$ from \eqref{holdersuneq1} in \eqref{holdersuneq2}, we get
\begin{equation*}
|\Omega\setminus H_t|^{\frac{d-1}{d}}\le 2 C_d P(H_t;\Omega)+\Lambda CP(H_t;\Omega)^\beta+\Lambda C|\Omega\setminus H_t|^\beta+\mu C_d|\Omega\setminus H_t|.
\end{equation*}
Now setting $\vf(t):=|\Omega\setminus H_t|$, we have that $\vf'(t)=-P(H_t;\Omega)$. Taking in consideration the fact that $\vf(t)\to 0$ and $|\vf'(t)|\le P(\Omega;H_t^c)\to 0$, as $t\to\infty$, we get that for some large $t_0$ we have
\begin{equation*}\label{holdersuneq3}
\vf(t)^{\frac{d-1}{\beta d}}\le -\Lambda^{\frac1\beta} C \vf'(t),\qquad \forall t\ge t_0,
\end{equation*}
where $C$ depends on $d$, $\beta$ and $|\Omega|$. Then, since $\ds\frac{d-1}{\beta d}<1$ and $\vf(t_0)\le |\Omega|$,  we have
$$\vf(t)\le \Big(|\Omega|^{1-\frac{d-1}{\beta d}}-\big(1-\frac{d-1}{\beta d}\big)\Lambda^{\frac1\beta}C(t-t_0)\Big)^{\frac1{1-\frac{d-1}{\beta d}}},$$
and so, $\vf(t)=0$, for $t\ge t_0+C\Lambda^{-1/\beta}$, where $C$ depends on $d$, $\beta$ and $|\Omega|$. Repeating the argument in every direction, we obtain the boundedness of $\Omega$. 
\end{proof}

\subsection{Interior quasi-minimality for subsolutions with Lipschitz energy function}

The following lemma was proved by Alt and Caffarelli \cite{altcaf} for harmonic functions and is implicitly contained in \cite{bucve}. The more general statement for capacitary measures can be found in \cite{tesi}.

\begin{lemma}\label{ballcut}
Let $\Omega\subset\R^d$ be a set of finite measure. Then there exist constants $C_d$, depending only on the dimension $d$ such that, for each ball $B_r(x_0)\subset\R^d$, we have the following estimate for the energy function $w_\Omega$.
\begin{equation*}\label{sopra1}
\widetilde E_1\big(\Omega\setminus B_r(x_0)\big)-\widetilde E_1(\Omega)\le  C_d\left(r+\frac{\|w_\Omega\|_{L^\infty(B_{2r}(x_0))}}{2r}\right)\int_{\partial B_r(x_0)}w_\Omega\,d\HH^{d-1}+\int_{B_r(x_0)}w_\Omega\,dx.
\end{equation*}
\end{lemma}
This estimate leads to the following result:
\begin{lemma}\label{subconclusivelemma06}
Consider an open set $\Omega\subset\R^d$ such that
\begin{equation}\label{holdersub2}
P(\Omega)-P( U)\le \Lambda\Big(\widetilde E_1( U)-\widetilde E_1(\Omega)\Big)^\beta+\mu|\Omega\setminus U|,\qquad \forall  U\subset\Omega\quad\hbox{such that}\quad \widetilde E_1( U)-\widetilde E_1(\Omega)\le \eps,
\end{equation}
where $\Lambda>0$, $\beta>0$, $\eps>0$ and $\mu\ge 0$ are given constants. If the energy function $w_\Omega\in H^1(\R^d)$ is Lipschitz continuous on $\R^d$, then there are constants $r_0>0$, depending on $d$, $\eps$, $|\Omega|$, and $C>0$ such that the following interior quasi-minimality condition holds:
$$P(\Omega)\le P( U)+Cr^{d\beta}+\mu \omega_d r^d,\qquad \forall  U\subset\Omega\quad\hbox{such that}\quad \Omega\setminus U\subset B_r(x_0)\quad\hbox{for some}\quad x_0\in\partial\Omega.$$
%A precise account on the constant above is $C=\Lambda\big(d\omega_d(r_0+L) L+\omega_d r_0 L\big)^\beta$, where $L$ is the Lipschitz constant of $w_{\Om}$.
\end{lemma}
\begin{proof}
By Lemma \ref{ballcut}, for $r_0$ small enough, we have $\widetilde E_1\big(\Omega\setminus B_{r_0}(x_0)\big)-\widetilde E_1(\Omega)\le\eps$. Thus, we can use any set $ U\subset\Omega$, such that $\Omega\setminus U\subset  B_r(x_0)\subset B_{r_0}(x_0)$, to test \eqref{holdersub2}. Indeed, we have 
\begin{equation*}
\begin{array}{ll}
\ds P(\Omega)-P( U)&\ds \le \Lambda\Big(\widetilde E_1( U)-\widetilde E_1(\Omega)\Big)^\beta+\mu|\Omega\setminus U|\\
\\
&\ds \le\Lambda\Big(\widetilde E_1\big(\Omega\setminus B_r(x_0)\big)-\widetilde E_1(\Omega)\Big)^\beta+\mu|\Omega\cap B_r(x_0)|\\
\\
&\ds \le \Lambda\left[\left(r+\frac{\|w_\Omega\|_{L^\infty(B_{2r}(x_0))}}{2r}\right)\int_{\partial B_r(x_0)}w_\Omega\,d\HH^{d-1}+\int_{B_r(x_0)}w_\Omega\,dx\right]^\beta
+\mu\omega_dr^d\\
\\
&\ds\le \Lambda\big(d\omega_d(r_0+L) L+\omega_d r_0 L\big)^\beta r^{d\beta}+\mu\omega_dr^d,
\end{array}
\end{equation*} 
where $L$ is the Lipschitz constant of $w_\Omega$.
\end{proof}

%\begin{prop}
%Suppose that $\Omega\subset\R^d$ is an open set of finite Lebesgue measure such that the energy function $w_\Omega$ is Lipschitz continuous on $\R^d$. If $\Omega$ is a subsolution for the functional $\widetilde E_f+P+\mu|\cdot|$, where $f\in L^p(\R^d)$ for $p>d$, then the following interior quasi-minimality condition holds:
%$$P(\Omega)\le P(\omega)+Cr^{\frac{(p-1)d}{p}},\qquad \forall \omega\subset\Omega\quad\hbox{such that}\quad \Omega\setminus\omega\subset B_r(x_0),\quad\forall x_0 \in\partial\Omega,\ \forall r<1,$$
%where the constant $C$ depends on the dimension $d$, the exponent $p$, the measure $|\Omega|$, the norm $\|f\|_{L^p}$, the constant $\mu\in\R$ and the Lipschitz constant of $w_\Omega$.
%\end{prop}

%\begin{prop}
%Suppose that $\Omega\subset\R^d$ is an open set of finite Lebesgue measure such that the energy function $w_\Omega$ is Lipschitz continuous on $\R^d$. If $\Omega$ is a subsolution for the functional $\mathcal{F}=F\big(\widetilde\lambda_1,\dots,\widetilde\lambda_k\big)+P+\mu|\cdot|$, where $F:\R^k\to\R$ is H\"older continuous with exponent $\beta>\frac{d-1}{d}$, then the following interior quasi-minimality condition holds:
%$$P(\Omega)\le P(\omega)+Cr^{d\beta},\qquad \forall \omega\subset\Omega\quad\hbox{such that}\quad \Omega\setminus\omega\subset B_r(x_0),\quad\forall x_0 \in\partial\Omega,\ \forall r<r_0,$$
%where $C$ and $r_0$ depend on the dimension $d$, the function $F$, the eigenvalue $\widetilde\lambda_k(\Omega)$, the measure $|\Omega|$, the constant $\mu\in\R$ and the Lipschitz constant of $w_\Omega$.
%\end{prop}

\section{Proof of Theorem \ref{th:mainD}}\label{sect:final}
Consider first the shape optimization problem 
\begin{equation}\label{optGtilde06}
\min\Big\{P(\Omega)+\widetilde{\mathcal{G}}(\Omega)\ :\ \Omega\subset\Dr\text{ measurable},\ |\Omega|=m\Big\},
\end{equation}
where the \emph{box} $\Dr$ is a bounded open set in $\R^d$ with $C^2$ boundary or $\Dr=\R^d$. 
\begin{enumerate}
\item By Proposition \ref{prop0}, there is a solution $\Omega^\ast$ of \eqref{optGtilde06}.\\[-3mm]
\item By Proposition \ref{prop1}, there are some $\Lambda>0$ and $r>0$ such that $\Omega^\ast$ is a solution of the problem 
\begin{equation*}\label{optGtilde06tt}
\min\Big\{P(\Omega)+\widetilde{\mathcal{G}}(\Omega)+\Lambda\big||\Omega|-m\big|\ :\ \Omega\subset\Dr\text{ measurable},\ \text{diam}(\Omega\Delta\Omega^\ast)\le r\Big\}. 
\end{equation*}
\item In particular, due to the monotonicity of $\widetilde{\mathcal{G}}$ with respect to the set inclusion, we have that, for every $\Omega\subset\Dr$ with $\Omega^\ast\subset\Omega$
$$P(\Omega^\ast)+\Lambda |\Omega^\ast|\le P(\Omega)+\Lambda|\Omega|,$$
i.e. the conditions of Proposition \ref{prop2} are satisfied and in particular the torsion function $w_{\Omega^\ast}$ is Lipschitz continuous.\\[-3mm]
\item On the other hand, if $\Omega\subset\Omega^\ast$, then 
$$P(\Omega^\ast)+\widetilde{\mathcal{G}}(\Omega^\ast)-\Lambda |\Omega^\ast|\le P(\Omega)+\widetilde{\mathcal{G}}(\Omega)-\Lambda|\Omega|.$$
Then the conditions of Proposition \ref{prop3} are satisfied and, in particular, $\Omega^\ast$ is a local interior quasi-minimizer of the perimeter with some exponent $d\beta\in(d-1,d]$.\\[-3mm]
\item Thus, for a generic set $\Omega$ such that $\Omega\Delta\Omega^\ast$ is contained in a ball of sufficiently small radius $r>0$, we can apply Proposition \ref{prop2} and Proposition \ref{prop3} and obtain that
\begin{align*}
P(\Omega)&\ge P(\Omega\cap\Omega^\ast)+P(\Omega\cup\Omega^\ast)-P(\Omega^\ast)\\
&\ge P(\Omega^\ast) - Cr^{d\beta} + P(\Omega^\ast)- \Lambda |B_r|-P(\Omega^\ast)\ge P(\Omega^\ast)-Cr^ {d\beta}.
\end{align*}
This proves that  $\Omega^\ast$ is a local quasi-minimizer of the perimeter. Therefore,  it is bounded and satisfies the regularity property \eqref{eq:reg}. In particular 
$\widetilde{\mathcal{G}}(\Omega^\ast)=\mathcal{G}(\Omega^\ast)$.\\[-3mm]
\item Let now $\Omega$ be a generic open set of measure $m$ in $\Dr$. Then we have
\begin{equation*}\label{omegaomegaast06}
P(\Omega)+\mathcal{G}(\Omega)\ge P(\Omega)+\widetilde{\mathcal{G}}(\Omega)\ge P(\Omega^\ast)+\widetilde{\mathcal{G}}(\Omega^\ast)=P(\Omega^\ast)+{\mathcal{G}}(\Omega^\ast),
\end{equation*}
which proves that $\Omega^\ast$ is a solution of \eqref{eq:P+Ef3}.

On the other hand, if $\Omega$ is another solution of \eqref{eq:P+Ef3}, then  
\begin{equation*}\label{omegaomegaast062}
P(\Omega)+\mathcal{G}(\Omega)\le P(\Omega^\ast)+{\mathcal{G}}(\Omega^\ast)= P(\Omega^\ast)+\widetilde{\mathcal{G}}(\Omega^\ast)\le P(\Omega)+\widetilde{\mathcal{G}}(\Omega)\le P(\Omega)+\mathcal{G}(\Omega),
\end{equation*}
which proves that $\Omega$ is also a solution of \eqref{optGtilde06}. Therefore,  it  enjoys the same regularity properties of  $\Om^*$.
% In particular $\Omega$ is a quasi-minimizer of the perimeter and so it is a bounded open set that satisfies the regularity properties \eqref{eq:reg}.
\end{enumerate}\qed\\

%\section{Further remarks and examples}
%\begin{itemize}
%\item Nonexistence of optimal set for $E_f+P$ and $\Dr$ convex unbounded.
%\item Existence of optimal set  for small masses for $E_f+P$ and $\Dr$ convex . 
%\end{itemize}
\noindent{\bf Aknowledgement:} This work was supported by the projects ANR-12-BS01-0007 OPTIFORM financed by the French Agence Nationale de la Recherche (ANR). G.D.P. is supported by the MIUR SIR-grant ``Geometric Variational Problems'' (RBSI14RVEZ). G.D.P is a  member of the ``Gruppo Nazionale per l'Analisi Matematica, la Probabilit\'a e le loro Applicazioni'' (GNAMPA) of the Istituto Nazionale di Alta Matematica (INdAM).
%%%%%%%%%%%%%%%%%%%%%%%%%%%%%%%%%%%%%%%%%%%%%%%%%%


\begin{thebibliography}{999}

%\bibitem{AFM}\textsc{Acerbi E., Fusco N., Morini M.},  \emph{Minimality via second variation for a nonlocal isoperimetric problem},  Comm. Math. Phys. 322 (2013), no. 2, 515-557

\bib{altcaf}{H.W.~Alt, L.A.~Caffarelli}{Existence and regularity for a minimum problem with free boundary}{J. Reine Angew. Math. {\bf325} (1981), 105-144}

\bib{amfupa}{L. Ambrosio, N. Fusco, D. Pallara}{Functions of Bounded Variation and Free Discontinuity Problems}{Oxford mathematical monographs, 2000}

\bib{ACKS}{I. Athanasopoulos, L. A. Caffarelli, C. Kenig, S. Salsa}{An area-Dirichlet
integral minimization problem}{Communications in Pure and Applied Mathematics,
54(4) :479-499, 2001}

\bib{MvDB}{M. van den Berg} {On the minimization of Dirichlet eigenvalues.} {Bull. Lond. Math. Soc. 47, 143-155, 2015}.

\bib{bogosel}{B.~Bogosel}{Regularity result for a shape optimization problem under
perimeter constraint}{Preprint available at https://arxiv.org/abs/1606.05878}

\bib{bogoud}{B.~Bogosel, E.~Oudet}{Qualitative and Numerical Analysis of a Spectral Problem with Perimeter Constraint} {SIAM J. Control Optim. 54 (2016), no. 1, 317-340}

\bib{BHP}{T.~Brian\c{c}on, M. Hayouni, M. Pierre}{Lipschitz continuity of state functions in some
optimal shaping}{Calc. Var. PDE, {23} (1) (2005), 13-32}

\bib{briancon}{T. Brian\c{c}on}{Regularity of optimal shapes for the Dirichlet's energy with volume constraint}
{ESAIM: COCV, Vol. 10 (2004), 99-122}

\bib{brla}{T.~Brian\c{c}on, J.~Lamboley}{Regularity of the optimal shape for the first eigenvalue of the Laplacian with volume and inclusion constraints}{Ann. Inst. H. Poincar\'e Anal. Non Lin\'eaire {\bf26} (4) (2009), 1149-1163}

\bib{bresib}{H. Br\'ezis, M. Sibony}{Equivalence de deux in\'equations variationnelles et applications}{Arch. Rational Mech. Anal. 41 (1971), 254-265}

\bib{bulbk}{D.~Bucur}{Minimization of the k-th eigenvalue of the Dirichlet Laplacian}{Arch. Rational Mech. Anal. {\bf 206} (3) (2012), 1073-1083}

\bib{buc00}{D.~Bucur}{Uniform concentration-compactness for Sobolev spaces
on variable domains}{Journal of Differential Equations {\bf162} (2000), 427-450}

\bib{bubu05}{D.~Bucur, G.~Buttazzo}{Variational Methods in Shape Optimization Problems}{Progress in Nonlinear Differential Equations {\bf65}, Birkh\"auser Verlag, Basel (2005)}

\bib{bubuhe}{D.~Bucur, G.~Buttazzo, A~Henrot}{Minimization of $\lambda_{2}$ with a perimeter constraint}{Indiana Univ. Math. J., 58, no. 6, 2709-2728, 2009}

 \bib{bubuve2}{D. Bucur - G. Buttazzo - B. Velichkov}{Spectral optimization problems for potentials and measures}{ SIAM J. Math. Anal., 46 (2014), no. 4, 2956-2986}

\bib{bucve}{D.~Bucur, B.~Velichkov}{Multiphase shape optimization problems}{SIAM J. Control Optim. 52 (2014), no. 6, 3556-3591}

\bib{bucmaprave}{D.~Bucur, D. Mazzoleni, A. Pratelli, B.~Velichkov}{Lipschitz regularity of the eigenfunctions on optimal domains}{Arch. Ration. Mech. Anal., 216 (2015), no. 1, 117-151. }
 
\bib{BL}
{\sc V. I. Burenkov, P. D. Lamberti}
{Spectral stability of Dirichlet second order
uniformly elliptic operators}{J. Differential Equations 244 (2008) 1712-1740}

%\bib{but10}{G.~Buttazzo}{Spectral optimization problems}{Rev. Mat. Complut. {\bf24} (2) (2011), 277-322}


%\bib{bdm93}{G.~Buttazzo, G.~Dal Maso}{Shape optimization for Dirichlet problems: relaxed solutions and optimality conditions}{Bull. Amer. Math. Soc., {\bf23} (1990), 531-535}

%\bib{buve}{G.~Buttazzo, B.~Velichkov}{Shape optimization problems on metric measure spaces}{J.Funct.Anal. {\bf 264} (1) (2013), 1-33}




%\bib{caffri85}{L.A.~Caffarelli, A.~Friedman}{Regularity of the boundary of a capillary drop on an inhomogeneous plane and related variational problems}{Rev. Mat. Iberoamericana, {\bf1} (1) (1985), 61-84}

\bib{CJK}{L.A. Caffarelli - D. Jerison - C.E. Kenig}{Some New Monotonicity Theorems
with Applications to Free Boundary Problems}{Annals of Math., 155 (2002), 369-404}


\bib{deve}{G.~De~Philippis, B.~Velichkov}{Existence and regularity of minimizers for some spectral optimization problems with perimeter constraint}{Appl. Math. Optim., 2014, Volume 69, Issue 2, pp 199-231}


\bib{Des}{J.~Descloux}{A stability result for the magnetic shaping problem}{Z. Angew. Math. Phys. 45 (1994), 544-555}
%\bib{dm93}{G.~Dal Maso}{An introduction to $\Gamma$-convergence}{Progress in Nonlinear Differential Equations and their Applications (PNLDE) {\bf8}, Birkh\"auser-Verlag, Basel (1993)}

%\bib{EF}{L. Esposito, N. Fusco}{A remark on a free interface problem with volume constraint}
%{J. Convex Anal. \textbf{18} (2011), 417-426}

%\bib{evgar}{L.~Evans, R.~Gariepy}{Measure Theory and Fine Properties of Functions}{Studies in Advanced mathematics, Crc Press (1991)}

%\bib{finn80}{R.~Finn}{The sessile liquid drop. I. Symmetric case}{Pacific J. Math. {\bf88} (2) (1980), 541-587}

\bib{giusti}{E.~Giusti}{Minimal Surfaces and Functions of Bounded Variation}{Springer, {1984}}

%\bib{giusti81}{E.~Giusti}{The equilibrium configuration of liquid drops}{J. Reine Angew. Math., {\bf321} (1981), 53-63}

%\bib{gonzalez76}{E.H.A.~Gonz\`alez}{Sul problema della goccia appoggiata}
{Rend. Semin. Mat. Univ. Padova, {\bf55} (1976), 289-302}

\bib{GMT83}{E.H.A.~Gonz\`alez, U. Massari, I. Tamanini}{On the Regularity of Boundaries of Sets Minimizing Perimeter with a Volume Constraint}
{Indiana U. Math. J., {\bf32} 1 (1983), 25-37}


\bib{hepi05}{A.~Henrot, M.~Pierre}{Variation et Optimisation de Formes. Une Analyse G\'eom\'etrique}{Math\'ematiques \& Applications {\bf48}, Springer-Verlag, Berlin (2005)}

\bib{landais}{N. Landais }{A Regularity Result in a Shape Optimization Problem with Perimeter}{Journal of Convex Analysis 14 (2007), No. 4, 785-806}

\bib{landais2}{N. Landais}{H\"older Continuity in a Shape-Optimization Problem with Perimeter}{Diff and Int. Equ., Vol. 20, Nt 6 (2007), 657-670.}

\bib{maggi}{F.~Maggi}{Sets of finite perimeter and geometric variational problems: an introduction to Geometric Measure Theory}{Cambridge University Press {\bf 135} (2012)}

\bib{mp}{D.~Mazzoleni, A.~Pratelli}{Existence of minimizers for spectral problems}{J. Math. Pures Appl. {\bf 100} (3) (2013), 433-453}

\bib{talenti}{G. Talenti}{Elliptic equations and rearrangements}{Ann. Scuola Normale Superiore di Pisa {\bf 3} (4) (1976), 697--718}

\bib{tamanini}{I. Tamanini}{Variational problems of least area type with constraints}{Annals of the University of Ferrara {\bf34} (1988), 183-217}

\bib{tesi}{B.~Velichkov}{Existence and regularity results for some shape optimization problems}{Volume 19, Edizioni della Normale, Pisa, 2015}

%\bib{wente80}{H.C.~Wente}{The symmetry of sessile and pendent drops}{Pacific J. Math. {\bf88} (2) (1980), 387-397}

\end{thebibliography}
\end{document}